\documentclass[12pt]{amsart}
\author{Ofir Gorodetsky} 
\address{Mathematical Institute, University of Oxford, Oxford, OX2 6GG, UK}
\email{gorodetsky@maths.ox.ac.uk}
\title{Smooth numbers and the Dickman $\rho$ function}
\thanks{This project has received funding from the European Research Council (ERC) under the European Union's Horizon 2020 research and innovation programme (grant agreement No 851318).}

\usepackage{amssymb}
\usepackage{amsfonts}
\usepackage{amsthm}
\usepackage{amsmath}
\usepackage{mathtools}
\usepackage{hyperref}

\theoremstyle{plain}
\newtheorem{thm}{Theorem}[section]

\newtheorem{lem}[thm]{Lemma}  
\newtheorem{proposition}[thm]{Proposition}
\newtheorem{cor}[thm]{Corollary}

\theoremstyle{remark}
\newtheorem{remark}{Remark}

\theoremstyle{definition}
\newtheorem*{acknowledgement}{Acknowledgements}

\newcommand{\CC}{\mathbb{C}}
\newcommand{\NN}{\mathbb{N}}
\newcommand{\RR}{\mathbb{R}}
\newcommand{\ZZ}{\mathbb{Z}}

\newcommand{\secondsad}{\beta}
\newcommand{\suprho}{\vartheta}
\newcommand{\bddfunc}{A}

\mathtoolsset{showonlyrefs}

\numberwithin{equation}{section}

\dedicatory{To Professor Peter Sarnak on his 70th birthday}

\begin{document}
\begin{abstract}
We establish an asymptotic formula for $\Psi(x,y)$ whose shape is $x \rho(\log x/\log y)$ times correction factors. These factors take into account the contributions of zeta zeros and prime powers and the formula can be regarded as an (approximate) explicit formula for $\Psi(x,y)$. With this formula at hand we prove oscillation results for $\Psi(x,y)$, which resolve a question of Hildebrand on the range of validity of $\Psi(x,y) \asymp x\rho(\log x/\log y)$. We also address a question of Pomerance on the range of validity of $\Psi(x,y) \ge x \rho(\log x/\log y)$. 

Along the way we improve classical estimates for $\Psi(x,y)$ and, on the Riemann Hypothesis, uncover an unexpected phase transition of $\Psi(x,y)$ at $y=(\log x)^{3/2+o(1)}$.
\end{abstract}
	\maketitle

\section{Introduction}
A positive integer is called $y$-smooth if each of its prime factors does not exceed $y$. We denote the number of $y$-smooth integers not exceeding $x$ by $\Psi(x,y)$. We assume throughout $x \ge y \ge 2$.
Let $\rho\colon [0,\infty) \to (0,\infty)$ be the Dickman function, defined as $\rho(t)=1$ for $t\le 1$ and via the delay differential equation $t\rho'(t) = -\rho(t-1)$ for $t>1$.
Dickman \cite{dickman1930} showed that
\begin{equation}\label{eq:dickman} 
\Psi(x,y) \sim x \rho( \log x / \log y) \qquad (x \to \infty)
\end{equation}
holds when $y \ge x^{\varepsilon}$. For this reason, it is useful to introduce the parameter
\[ u := \log x /\log y.\]
The range of validity of \eqref{eq:dickman} was considerably improved by de Bruijn \cite{debruijn1951} and H.~Maier (unpublished), and the state of the art is due to Hildebrand who showed that \cite{Hildebrand1986}
\begin{equation}\label{eq:debruijn}
	\Psi(x,y) = x\rho(u) \left(1+O_{\varepsilon}\left( \frac{\log (u+1)}{\log y}\right)\right)
\end{equation}
holds when $\log y \ge (\log \log x)^{\frac{5}{3}+\varepsilon}$. In \cite{Hildebrand1984}, Hildebrand showed that the Riemann Hypothesis (RH) implies that \eqref{eq:debruijn} holds in
\begin{equation}\label{eq:rhrange}
y \ge (\log x)^{2+\varepsilon}.
\end{equation}
In fact, in this range he proves the slightly stronger estimate
\begin{equation}\label{eq:stronger}
\Psi(x,y) = x\rho(u) \exp\left(O_{\varepsilon}\left( \frac{\log (u+1)}{\log y}\right)\right).
\end{equation}
The reverse implication is also true: if even the weaker estimate $\Psi(x,y)=x \rho(u)\exp(O_{\varepsilon}(y^{\varepsilon}))$ holds in the range \eqref{eq:rhrange} for every $\varepsilon>0$ then RH must be true. In \cite[p.~290]{Hildebrand1986}, Hildebrand speculates that $\Psi(x,y) \asymp x \rho(u)$ does not hold for $y \le (\log x)^{2-\varepsilon}$. He writes 
\begin{quote}\label{eq:quotehild}
	If the Riemann hypothesis is assumed, the range for $u$ can be further extended to $ 1\le u \le \log x/(2+\varepsilon)\log \log x$, but it seems likely that then the critical limit is attained: it may be conjectured that for $u>\log x/(2-\varepsilon)\log \log x$, the relation $\Psi(x,x^{1/u}) \sim x \rho(u)$ no longer holds.
\end{quote}
This conjecture is repeated in \cite{HildebrandLocal} and by Granville in \cite{Granville1989,Granville1993}. We confirm it in stronger form. To state our result we recall the function
\[ K \colon (-1,0] \to \RR, \qquad K(t) = t\zeta(t+1)/(t+1), \qquad K(0)=1,\]
was introduced by de Bruijn \cite[Eq.~(2.8)]{debruijn1951}. Evidently $\lim_{t \to -1^+} K(t)= \infty$. It is strictly decreasing: the identity $\zeta(s)=s/(s-1) -s \int_{1}^{\infty}\{x\}x^{-1-s}dx$ for $s >0$ \cite[Eq.~(1.24)]{MV} implies $K(t) = 1-t\int_{1}^{\infty}\{x\}x^{-2-t}dx$ and $K'(t)=-\int_{1}^{\infty}\{x\}(1-t\log x)x^{-2-t}dx<0$.
\begin{thm}\label{thm:hildconjstronger} Let $\suprho := \sup_{\zeta(\rho)=0} \Re \rho \in [1/2,1]$ be supremum of the real parts of the zeros of the Riemann zeta function. Fix $\varepsilon>0$. 
\begin{enumerate}
	\item Suppose $\suprho\neq 1$. If $x\ge y \ge (\log x)^{1/(1-\suprho)+\varepsilon}$ then
\begin{equation}\label{eq:rhoasymp}
\Psi(x,y) \sim x \rho(u) K(-\log \log x/\log y), \qquad x \to \infty.
\end{equation}
\item Given $A \in (1/\suprho,1/(1-\suprho))$\footnote{If $\suprho=1$ we define $1/(1-\suprho):=\infty$.} and $\mathrm{sgn} \in \{+,-\}$ there exists an explicit function $x=x(y)$ of $y$ satisfying $y=(\log x)^{A+o(1)}$ and
\begin{equation}\label{eq:omegaplus}
\Psi(x,y) = x \rho(u) \exp(\Omega_{\mathrm{sgn} }(y^{\suprho+A^{-1}-1-\varepsilon})) .
\end{equation}
\item Suppose $\suprho \neq 1$. If $2 \log x\le y\le (\log x)^{1/\suprho-\varepsilon}$ and $x \ge C_{\varepsilon}$ then
\begin{equation}\label{eq:3rdthm}
\Psi(x,y) = x \rho(u)  \exp\bigg(\Theta_{\varepsilon}\bigg( \frac{ \log^2 x}{y \log y} \bigg)\bigg).
\end{equation}
\end{enumerate}
\end{thm}
The key input to \eqref{eq:omegaplus} is Landau's Oscillation Theorem. See Remark \ref{rem:omega} for a refinement.

In \cite{Granville2008,Lichtman}, Pomerance asked whether $\Psi(x,y) \ge x \rho(u)$
holds for all $x/2 \ge y \ge 1$. 
In \cite[p.~274]{Granville2008}, Granville proved that $\Psi(x,y) > 2x\rho(u)$ holds for $y \le c(\log x)(\log \log x)/\log \log \log x$ if $x \ge C$.
Below we extend Granville's range considerably. Moreover, if RH is true, we show $\Psi(x,y) \ge x\rho(u)$ holds when $y \not\in [(\log x)^{2-\varepsilon}, (\log x)^{2+\varepsilon}]$. For $y$ near $(\log x)^2$, the question lies beyond RH in a precise sense, but we indicate that a positive answer follows from a conjecture of Montgomery and Vaughan on the size of the remainder term in the PNT. If RH is false, the $\Omega_{-}$ result in \eqref{eq:omegaplus} already implies the inequality fails infinitely often.

Let
\[ L := \max_{v \in \RR} e^v \left( -\log (-\zeta(1/2)) - \frac{1}{2}\int_{v}^{2v} e^{-r}r^{-1}dr \right)\approx -0.666217.\]
\begin{thm}\label{thm:pom}Fix $\varepsilon>0$ and suppose $x \ge C_{\varepsilon}$.
\begin{enumerate}
	\item For $y \in [\exp((\log \log x)^{5/3 +\varepsilon}),(1-\varepsilon) x]$ we have
	\begin{equation}\label{eq:smallu}
	\Psi(x,y) = x\rho(u) (1+\Theta_{\varepsilon} (\log (u+1)/\log y)) \ge x\rho(u)
\end{equation}
	and for $y \in [\varepsilon\log x,\exp((\log \log x)^{3/5-\varepsilon})\log x]$ 
	\begin{equation}\label{eq:granvilleextended}
	\Psi(x,y) = x \rho(u) \exp\left(\Theta_{\varepsilon}\left( \frac{\log^2 x}{y \log y}\right)\right) \ge x \rho(u).
\end{equation}
	\item Suppose RH is true. Then \eqref{eq:smallu} holds for $y \in [(\log x)^{2+\varepsilon}, (1-\varepsilon)x]$ and \eqref{eq:granvilleextended} holds for $y \in [\varepsilon\log x, (\log x)^{2-\varepsilon}]$. If $\Psi(x,y)\ge x\rho(u)$ holds for $y\in[(\log x)^{3/2},(\log x)^{3}]$ then
	\begin{equation}\label{eq:necc}
	\liminf_{y \to \infty} \frac{\psi(y)-y}{\sqrt{y}\log y} \ge L.
\end{equation}
	If \eqref{eq:necc} holds with strict inequality then $\Psi(x,y)\ge x\rho(u)$ holds for $y\in [2,(1-\varepsilon)x]$.
\end{enumerate}
\end{thm}
RH implies $\psi(y)-y \ll\sqrt{y}(\log y)^2$ \cite[Thm.~13.1]{MV}. It is believed that 
\begin{equation}\label{eq:munsch} \liminf_{y \to \infty} \frac{\psi(y)-y}{\sqrt{y}(\log \log \log y)^2} = -\frac{1}{2\pi},
\end{equation}
see \cite[p.~484]{MV}; \eqref{eq:munsch} implies that the limit considered in \eqref{eq:necc} is $0$. 
\subsection*{Conventions and notation}
The letters $C,c$ denote absolute positive constants that may change between different occurrences. The notation $A \ll B$ means $|A| \le C B$ for some absolute constant $C$, and $A\ll_{a,b,\ldots} B$ means $|A|\le C_{a,b,\ldots} B$ for $C_{a,b,\ldots}$ that may depend on the subscripts. We write $A \asymp B$ to mean $C_1 B \le A \le C_2 B$ for some absolute positive constants $C_i$, and $A \asymp_{a,b,\ldots} B$ means $C_i$ may depend on $a,b,\ldots$. We write $\Theta(B)$ and $\Theta_{a,b,\ldots}(B)$ to indicate a quantity $A$ with $A\asymp B$ and $A \asymp_{a,b,\ldots} B$, respectively. We write $\Omega_{+}(g(x))$ (resp.~$\Omega_{-}(g(x))$) to indicate a function $f(x)$ with $\limsup_{x\to \infty} f(x)/g(x)>0$ (resp.~$\liminf_{x \to \infty} f(x)/g(x)<0$). A function $f$ is $\Omega_{\pm}(g(x))$ if $\limsup f/g>0>\liminf f/g$. Throughout 
\[L(x) = \exp((\log x)^{\frac{3}{5}}( \log \log (x+1))^{-\frac{1}{5}}).\]
We denote
\[\suprho = \sup_{\zeta(\rho)=0} \Re \rho \in [1/2,1].\]
For $y\ge 2$ and $\Re s>0$ we define the partial zeta function
\[\zeta(s,y)=\prod_{p \le y} (1-p^{-s})^{-1}=\sum_{n \text{ is }y\text{-smooth}} n^{-s}.\]
We define $\xi\colon [1,\infty) \to [0,\infty)$ via \[e^{\xi(v)}=1+v\xi(v).\]
We define the entire function \[I(s)=\int_{0}^{s} \frac{e^v-1}{v}dv= \sum_{i \ge 1} \frac{s^i}{i!i}.\] 
We denote the Euler--Mascheroni constant by $\gamma$.
The Laplace transform of $\rho$ is given by \cite[Eq.~(1.9)]{debruijn19512}
\[\hat{\rho}(s) := \int_{0}^{\infty} \rho(v)e^{-sv}dv =\exp( \gamma + I(-s))\]
for $s \in \CC$. We define
\[	F(s,y)=\zeta(s)(s-1)F_2(s,y)\]
for $s \in \CC$	where
	\[	F_2(s,y) =  \hat{\rho}((s-1)\log y)\log y.\]
We write $\psi(x)=\sum_{n \le x} \Lambda(n)$ for the Chebyshev function.
 We set
\[ \bar{u}=\min\{y/\log y, u\}.\]
If a meromorphic function has a removable singularity (e.g.~$\zeta(s)(s-1)$ at $s=1$), we identify its value there with its limit.
When we differentiate a bivariate function (e.g.~$\zeta(s,y)$) we always do so with respect to the first variable. In sums and products over $p$, $p$ is understood to be prime. Throughout, $\varepsilon>0$ is an arbitrary fixed constant.
\section{A formula and its investigation}\label{sec:proofs}
\subsection{On two saddle points}
Let $f \colon \RR \to \RR$ be given by
\begin{equation}\label{eq:f}
	f(t):= t\log x + \log( F_2(t,y)).
\end{equation}
It satisfies
\begin{align}
	f'(t) &= \log x -(\log y) I'((1-t)\log y), \qquad f''(t)=  (\log y)^2I''((1-t)\log y).
\end{align}
The function $f$ is convex since $I''(-r)=r^{-2}e^{-r}(e^{r}-(r+1))>0$. We define 
\[ \secondsad=\secondsad(x,y) = 1-\frac{\xi(u)}{\log y}\]
for $u=\log x/\log y$. We often shorten $\xi(u)$ to $\xi$ when no confusion may arise. 
The function $f$ attains its global minimum at $\secondsad$ because it is readily verified that 
\begin{equation}\label{eq:betaderiv}
	f'(\secondsad)=0.
\end{equation}
The asymptotics of $\secondsad$ are well understood thanks to the next lemma.
\begin{lem}\label{lem:xilem}\cite[Lem.~1]{HildebrandTenenbaum1986}\cite[Lem.~4.5]{Smida}
	For $v \ge 3$ we have
	\begin{align}
		\label{eq:xiasymp}\xi(v) &= \log v + \log \log v + O(\log \log v/ \log v),\\
		\label{eq:derivxi}		I''(\xi(v))&=v(1+O(1/\log v)).
	\end{align}
\end{lem}
\begin{cor}\label{cor:sigmasize}Fix $\varepsilon>0$.
	\begin{enumerate}
		\item 
		If $x\ge y\ge (1+\varepsilon) \log x$ then
		\begin{equation}\label{eq:sigmasymp}
			\secondsad= \frac{\log\big( \frac{y}{\log x}\big)}{\log y} \bigg(1 + O_{\varepsilon}\bigg( \frac{\log \log (y+1)}{\log y}\bigg)\bigg).
		\end{equation}
		\item If $2 \log x \ge y \ge \varepsilon\log x$ then $\secondsad = O_{\varepsilon}(1/\log y)$.
		\item We have $\secondsad \le 1$. Equality occurs if and only if $y=x$. 
		\item We have $\secondsad \ge 0$ if and only if $y\ge 1+\log x$, and $\secondsad=0$ if and only if $y=1+\log x$.
	\end{enumerate}
\end{cor}
\begin{proof}
	If $u<3$ then \eqref{eq:sigmasymp} means $\secondsad=1+O(\log \log (x+1)/\log x)$ which follows from $\xi(u)=O(1)$. If $u\ge 3$ we use \eqref{eq:xiasymp} to write
	\[ \secondsad = \frac{\log (y/(u \log u)) + O(\log \log (u+1)/\log u)}{\log y}\]
	and reduce \eqref{eq:sigmasymp} to
	\begin{equation}\label{eq:red}
		\log \left(\frac{\log y}{\log u}\right) + O\left(\frac{\log \log (u+1)}{\log u}\right) \ll_{\varepsilon}\log\left(\frac{y}{\log x}\right)\frac{\log\log (y+1)}{\log y}.
	\end{equation}
	If $x\ll_{\varepsilon} 1$ then $y\ll_{\varepsilon} 1$ and \eqref{eq:sigmasymp} is trivial, so we may  assume $x \ge C_{\varepsilon}$.
	If $y \ge \log^{3/2} x$ then the right-hand side of \eqref{eq:red} is $\gg \log \log(y+1)$ and it suffices to show $\log \log u \ll_{\varepsilon} \log \log (y+1)$, which is clear. If $(1+\varepsilon)\log x\le y<\log^{3/2} x$, the right-hand side of \eqref{eq:red} is 
	\[ \gg_{\varepsilon} \log\left(\frac{y}{\log x}\right) \frac{\log \log \log x}{\log \log x}. \]
	Since $\log \log (u+1)/\log u \ll_{\varepsilon} \log \log \log x/\log \log x$, \eqref{eq:red} reduces further to
	\begin{equation}\label{eq:further2}
		\log\left(\frac{\log y}{\log u}\right) \ll_{\varepsilon}  \log\left(\frac{y}{\log x}\right) \frac{\log \log \log x}{\log \log x}.
	\end{equation}
	We write
	\[\frac{\log y}{\log u} = \frac{\log y}{\log \log x} \big(1 -\frac{\log \log y}{\log \log x}\big)^{-1} = \big(1+\frac{\log(y/\log x)}{\log \log x}\big)(1-\frac{\log \log y}{\log \log x}\big)^{-1}\]
	and \eqref{eq:further2} follows by using $\log(1+t)\ll |t|$ for $|t|\le 1/2$.
	The second part of the lemma is similar to the first and is left to the reader.
	The third part follows from $\xi$ being $0$ at $v=1$ and being strictly increasing. For the last part we need to solve $\secondsad \ge 0$, or $\log y\ge \xi(u)$. Since $\xi$ is strictly increasing, it suffices to solve $\log y= \xi(u)$. Exponentiating, this implies $y = e^{\xi(u)} = 1+u\xi(u) =  1+\log x$.
\end{proof}
Let $g \colon (0,\infty) \to \RR$ be given by 	
\[ 	g(t):= t\log x + \log \zeta(t,y).\]
It satisfies
\[ g'(t) = \log x- \sum_{p \le y} \frac{\log p}{p^{t}-1}, \qquad g''(t)=\sum_{p\le y} \frac{p^t (\log p)^2}{(p^t -1 )^2}.\]
The function $g$ is convex because $g''(t) >0$ for $t>0$. Since $\lim_{t\to 0^+}g(t)=\lim_{t\to \infty}g(t)=\infty$ it follows that $g$ has a global minimum, attained at a point 
\[\alpha=\alpha(x,y)\]
which must satisfy
\begin{equation}\label{eq:alphaderiv}
	g'(\alpha)=0.
\end{equation}
The following theorem of Hildebrand and Tenenbaum expresses the asymptotics of $\Psi(x,y)$ in terms of $\alpha$.
\begin{thm}\cite[Thms.~1, 2]{HildebrandTenenbaum1986}\label{thm:htsaddle}
	For $x \ge y \ge 2$ we have
	\begin{equation}\label{eq:HTsaddle} \Psi(x,y)=\frac{x^{\alpha}\zeta(\alpha,y)}{\alpha\sqrt{2\pi\phi_2(\alpha,y)}}(1 +O(\bar{u}^{-1}))
	\end{equation}
	where
	\begin{equation}\label{eq:phi2}
		\phi_2(\alpha,y):=\sum_{p \le y}\frac{p^{\alpha}(\log p)^2}{(p^{\alpha}-1)^2}= \left(1+\frac{\log x}{y} \right)(\log x)(\log y)(1+O(\log(1+\bar{u})^{-1})).
	\end{equation}
Additionally,
\begin{equation}\label{eq:alphasize}
	 \alpha = \frac{\log\big( 1+\frac{y}{\log x}\big)}{\log y}\bigg(1+O\bigg(\frac{\log \log (y+1)}{\log y}\bigg)\bigg).
	\end{equation}
\end{thm}
The points $\alpha$ and $\secondsad$ are known to be close:
\begin{lem}\label{lem:sigmastuff}
	For $x \ge y \ge \varepsilon \log x$ we have $\secondsad-\alpha = O_{\varepsilon}(1/\log y)$.
\end{lem}
\begin{proof}
	If $x \ge y > \log x$ this is in \cite[Eq.~(3.5)]{HildebrandTenenbaum1986}. If $\log x \ge y \ge \varepsilon \log x$ both $\secondsad$ and $\alpha$ are $O_{\varepsilon}(1/\log y)$ by Corollary \ref{cor:sigmasize} and \eqref{eq:alphasize}.
\end{proof}
The following lemma is useful in simplifying the order of magnitude of $y^{\alpha}$ and $y^{\secondsad}$.
\begin{lem}\label{lem:ytoapower}
	For $x \ge y \ge 2$ we have $y^{1-\secondsad}\asymp u\log(u+1)$. For $x \ge y \ge \varepsilon\log x$ we have $y^{1-\alpha} \asymp_{\varepsilon} u\log(u+1)$.
\end{lem}
\begin{proof}
	The estimate for $y^{1-\secondsad}$ follows from the definition of $\secondsad$ and Lemma \ref{lem:xilem}. The estimate for $y^{1-\alpha}$ follows from the estimate for $y^{1-\secondsad}$ by Lemma \ref{lem:sigmastuff}.
\end{proof}
\subsection{The function \texorpdfstring{$G$}{G}}\label{sec:int}
We introduce
\[ G(s,y):= \frac{\zeta(s,y)}{F(s,y)}.\]
The following identities are tautological in their domain of definition:
\begin{equation}\label{eq:taut}
	\begin{split} \frac{x^{\alpha}\zeta(\alpha,y)}{x^{\secondsad}F_2(\secondsad,y)} &= \zeta(\alpha)(\alpha-1)G(\alpha,y)\exp(f(\alpha)-f(\secondsad))\\
		&=\zeta(\secondsad)(\secondsad-1)G(\secondsad,y)\exp(g(\alpha)-g(\secondsad)).
	\end{split}
\end{equation}
Let
\[ B(x,y):= \sqrt{\frac{(\log y)^2 I''(\xi(u))}{\phi_2(\alpha,y)}}.\]
Alladi proved \cite[Eq.~(3.9)]{Alladi1982}
\[	\rho(u) = \frac{e^{-u\xi}\hat{\rho}(-\xi)}{\sqrt{2\pi I''(\xi(u))}}(1+O(u^{-1}))\]
which can be phrased as
	\begin{equation}\label{eq:rho and i}
	x\rho(u) = \frac{x^{\secondsad}F_2(\secondsad,y)}{\sqrt{2\pi (\log y)^2 I''(\xi(u))}}(1+O(u^{-1})).
\end{equation}
Dividing \eqref{eq:HTsaddle} by \eqref{eq:rho and i} and employing \eqref{eq:taut} we obtain
\begin{lem}[Asymptotic formula]\label{lem:first}
If $ x \ge y \ge 2$ then $f(\alpha)\ge f(\secondsad)$ and
\begin{equation}\label{eq:firstpartasymp}
\frac{\Psi(x,y)}{x\rho(u)}=K(\alpha-1)G(\alpha,y) \exp( f(\alpha)-f(\secondsad)) B(x,y) ( 1 +O( \bar{u}^{-1} )).
\end{equation}
If $x \ge y > 1+ \log x$ then $g(\secondsad) \ge g(\alpha)$ and 
\[\frac{\Psi(x,y)}{x\rho(u)} = K(\secondsad-1)G(\secondsad,y) \exp( g(\alpha)-g(\secondsad)) \frac{\secondsad}{\alpha}B(x,y) ( 1 +O( u^{-1} )).\]
\end{lem}
When $y/\log x \to \infty$, \eqref{eq:sigmasymp} and \eqref{eq:alphasize} show $\alpha/\secondsad \sim 1$ and $K(\alpha-1)\sim K(\secondsad-1) \sim K(-\log \log x/\log y)$. If further $u \to \infty$ then $B(x,y) \sim 1$ by \eqref{eq:phi2} and \eqref{eq:derivxi}. Lemma \ref{lem:first} then implies
\begin{equation}\label{eq:simp}
	G(\alpha,y)(1+o(1))\le \Psi(x,y)/(x\rho(u) K(-\log \log x/\log y)) \le G(\secondsad,y)(1+o(1))
\end{equation}
as long as $u$ and $y/\log x$ both tend to $\infty$.

The savings in \eqref{eq:HTsaddle} and \eqref{eq:rho and i} are sharp. We can show that the lower order terms within the error terms are close, which allows us to prove in \S\ref{sec:sharpen} the following
\begin{proposition}\label{prop:sharpen}
If $y>1+\log x$ then the error terms in Lemma \ref{lem:first} are $O(1/(\alpha \log x))$.
\end{proposition}
\begin{remark}
With more work, $O(1/(\alpha \log x))$ may be improved to an explicit term of size $\asymp 1/(\alpha \log x)$ and an error of size $O(1/(\alpha^2 (\log x) (\log y)))$.
\end{remark}
Working in a zero-free region of $\zeta$ we choose the following logarithms of $\zeta(s,y)$ and $F(s,y)$:
\begin{align}
	\log \zeta(s,y) &= \sum_{p \le y} -\log(1-p^{-s}) = \sum_{n \text{ is }y\text{-smooth}} \Lambda(n)/(n^{s}\log n),\\
	\log F(s,y) &=  \log \log y + \gamma + I((1-s)\log y)+\log (\zeta(s)(s-1))
\end{align}
where $\log(\zeta(s)(s-1))$ is chosen to be real-valued for $s>1$. In further using Lemma \ref{lem:first} it will be of crucial importance to split $G$ as $G_1 G_2$ where
\begin{align}
	\log G_1(s,y) &= \sum_{n \le y} \frac{\Lambda(n)}{n^s\log n } -(\log (\zeta(s)(s-1)) + \log \log y + \gamma + I((1-s)\log y)),\\
	\log G_2(s,y) &= \sum_{k \ge 2} \sum_{y^{1/k} <p \le y} \frac{p^{-ks}}{k}.
\end{align}
We compute the Mellin transform of $\log G_1(s,y)$ (Proposition \ref{prop:mellinlogg1}) and then, via Landau's Oscillation Theorem, obtain the following lemma, essential to the proof of Theorem \ref{thm:hildconjstronger}.
\begin{lem}\label{lem:osc}
Fix $s > -2$. Then, as $x \to \infty$, $\log G_1(s,x)=\Omega_{\pm} (x^{\suprho-s-\varepsilon})$.
\end{lem}
Lemma \ref{lem:osc} is proved in \S\ref{sec:osc}. For $s=0$ it goes back to Landau. For $s=1$ it is due to Diamond and Pintz \cite{Diamond} and our standard proof follows theirs. 

The function $\log G_1(s,y)$ and its derivatives with respect to $s$ can be expressed as sums over zeros of $\zeta$ (Corollaries \ref{cor:logG1s2}--\ref{cor:S1is2}). This allows us to prove in \S\ref{sec:estimates} the following 
\begin{lem}\label{lem:logg1size}
Fix $i \ge 0$. For $\varepsilon -2 \le s \le 1/\varepsilon$ and $x \ge 4$ we have
	\begin{align}
		\label{eq:zerofree}	(\log G_1)^{(i)}(s,x) &\ll_{i,\varepsilon} (\log x)^{i}  x^{1-s}  L(x)^{-c}\ll_{i,\varepsilon}  x^{1-s}  L(x)^{-c_i},\\
		\label{eq:logGpsi2}	(\log G_1)^{(i)}(s,x) &= (-1)^i  (\log x)^{i-1} 	
		\frac{	\psi(x)-x + O_{i,\varepsilon}(x^{\suprho})}{x^s} \ll_{i,\varepsilon} x^{\suprho-s} (\log x)^{i+1}.
	\end{align}
\end{lem}
Littlewood proved that $\psi(x)-x =\Omega_{\pm}(\sqrt{x}\log \log \log x)$ \cite[Thm.~15.11]{MV}. Applying \eqref{eq:logGpsi2} if $\suprho=1/2$ and Lemma \ref{lem:osc} otherwise, we get
\begin{cor}
Fix $s>-2$. Then, as $x \to \infty$, $\log G_1(s,x) = \Omega_{\pm}(x^{\frac{1}{2}-s}\log \log x/\log x)$.
\end{cor}
In \S\ref{sec:intform} we prove the following illuminating representations of $\log G_1$.
\begin{lem}\label{lem:integralform} For $x> 2$ and $s > \suprho$,
\begin{equation}\label{eq:firstcase}
\log G_1(s,x)  = -\int_{x}^{\infty} \frac{d(\psi(t)-t)}{t^s \log t}.
\end{equation}
For $x> 2$ and $s>-2$,
\[		\log G_1(s,x)  = \int_{2^-}^{x} \frac{d(\psi(t)-t)}{t^s \log t}+\int_{1}^{2} \frac{t^{-2}-t^{-s}}{\log t}dt+\int_{2}^{\infty} \frac{dt}{t^2 \log t}-\log (\zeta(s)(s-1)).\]
\end{lem}
One can show $\log G_2(s,x) \to 0$ when $s-1/2 \ge \varepsilon$ and $x \to \infty$. In \S\ref{sec:G2} we prove more:
\begin{proposition}\label{prop:g2size}
Fix $i \ge 0$. 	For $x \ge 2$ and $1 \ge s\ge \varepsilon/\log x$ we have
	\begin{align}\label{eq:mtg2}
		(\log G_2)^{(i)}(s,x)&=( 1+O_{\varepsilon,i}( L(x)^{-c}+x^{-s}))\frac{(-2)^i}{2}\int_{\sqrt{x}}^{x}(\log t)^{i-1}t^{-2s}dt\\
		&\asymp_{\varepsilon,i} \frac{(-\log x)^{i}  x^{\max\{1-2s,\frac{1}{2}-s\}}}{\max\{1,|s-1/2|\log x\}}.
\end{align}	
For $1/4 \ge s \ge \varepsilon/\log x$ we have
	\begin{equation}\label{eq:smalls} \log G_2(s,x)=(1+O_{\varepsilon}(L(x)^{-c}))\int_{\sqrt{x}}^{x} (-\log(1-t^{-s})-t^{-s}) \frac{dt}{\log t}.
	\end{equation}
\end{proposition}
\subsection{Proof of Theorem \ref{thm:hildconjstronger}}
\subsubsection{First part}
We may assume $u \to \infty$ due to \eqref{eq:dickman}. By \eqref{eq:simp} it suffices to show $G(\alpha,y),G(\secondsad,y) \sim 1$. When $y \ge (\log x)^{1/(1-\suprho)+\varepsilon}$, the bounds for $\log G_1$ and $\log G_2$ given in \eqref{eq:logGpsi2} and \eqref{eq:mtg2} imply this (we simplify $y^{-\alpha}$ and $y^{-\secondsad}$ using Lemma \ref{lem:ytoapower}).
\subsubsection{Second part} 
Fix $\secondsad_0 \in (0,1)$. Given $y \ge 2$ there is a unique $x=x(y)$ with $\secondsad(x,y)=\secondsad_0$. Using \eqref{eq:betaderiv} we see that $x$ is determined by 
\[  (y^{1-\secondsad_0}-1)/(1-\secondsad_0) = \log x,\]
and in particular $y=(\log x)^{1/(1-\secondsad_0)+o(1)}$. Applying \eqref{eq:simp} with our $x=x(y)$ (so that $\secondsad=\secondsad_0$),
\begin{equation}\label{eq:corappmin}
	\Psi(x,y)\ll x \rho(u)  \exp(\log G_1(\secondsad_0,y)+\log G_2(\secondsad_0,y)).
\end{equation}
For our fixed $\secondsad_0$, 
\begin{equation}\label{eq:2ndmainthm}
\log G_1 (\secondsad_0,y) = \Omega_{-} (y^{\suprho-\secondsad_0-\varepsilon})
\end{equation}
by Lemma \ref{lem:osc}. If $\secondsad_0 \in (1-\suprho,\suprho)$, the $\Omega_{-}$ result follows from \eqref{eq:corappmin}, \eqref{eq:2ndmainthm} and the bound for $\log G_2$ given in \eqref{eq:mtg2}.
We now prove the $\Omega_{+}$ result. Given $y \ge 2$ and fixed $\alpha_0 \in (0,1)$ there is a unique $x=x(y)$ with $\alpha(x,y)=\alpha_0$. Using \eqref{eq:alphaderiv} we see that $x$ is determined by 
\[ -\zeta'(\alpha_0,y)/\zeta(\alpha_0,y)=\sum_{p \le y}\log p/(p^{\alpha_0}-1)=\log x.\]
By \eqref{eq:alphasize}, $ y=(\log x)^{1/(1-\alpha_0)+o(1)}$. Applying \eqref{eq:simp} with our $x=x(y)$ (so that $\alpha=\alpha_0$),
\begin{equation}\label{eq:corappmax}
	\Psi(x,y) \gg x\rho(u) \exp(\log G_1(\alpha_0,y)+\log G_2(\alpha_0,y)) \ge \exp(\log G_1(\alpha_0,y)).
\end{equation}
For our fixed $\alpha_0$,
\begin{equation}\label{eq:2ndmaththm2}
 \log G_1(\alpha_0,y) = \Omega_{+} (y^{\suprho-\alpha_0-\varepsilon})
\end{equation}
by Lemma \ref{lem:osc}. If $\alpha_0 \in (0,\suprho)$, the $\Omega_{+}$ result follows from \eqref{eq:corappmax} and \eqref{eq:2ndmaththm2}.
\subsubsection{Third part}
Let $s \in \{\alpha,\secondsad\}$. We use \eqref{eq:logGpsi2} and \eqref{eq:mtg2} to see that $\log G_1(s,y) \ll y^{\suprho-s}\log y$ and $\log G_2(s,y) \asymp_{\varepsilon} y^{1-2s}/\log y \asymp_{\varepsilon} \log^2 x/(y \log y)$ whenever $(\log x)^{1/\suprho-\varepsilon} \ge y \ge 2 \log x$ and $x \ge C$. We apply these estimates to
\[  (1+o(1))K(\alpha-1) G(\alpha,y) B(x,y)\le \frac{\Psi(x,y)}{x\rho(u)} \le (1+o(1))K(\secondsad-1) G(\secondsad,y) B(x,y)\frac{\secondsad}{\alpha}\]
which follows from Lemma \ref{lem:first}. The logarithms of $B(x,y)$, $K(s-1)$ and $\secondsad/\alpha$ are negligible: they are $\ll_{\varepsilon} \log \log y$ as seen from \eqref{eq:phi2}, \eqref{eq:derivxi} (for $B$) and \eqref{eq:alphasize}, \eqref{eq:sigmasymp} (for $K$ and $\secondsad/\alpha$).
\begin{remark}\label{rem:omega}
Estimates  \eqref{eq:corappmax} and \eqref{eq:2ndmaththm2} show that the $\Omega_{+}$ case of \eqref{eq:omegaplus} can be replaced with
\[ \Psi(x,y) \gg x \rho(u) \exp(\Omega_{+}(y^{\suprho+A^{-1}-1-\varepsilon}) + \log G_2(1-A^{-1},y)) \]
which holds for \textit{any} $A \in (1,1/(1-\suprho))$. In proving the $\Omega_{-}$ case of \eqref{eq:omegaplus} one may avoid using \eqref{eq:HTsaddle} by utilizing Rankin's bound $\Psi(x,y) \le x^{\alpha}\zeta(\alpha,y)$. 
\end{remark}
\section{New estimates and their proofs}
\begin{cor}\label{cor:strip}
Suppose $\suprho<1$. If $x \ge y \ge (\log x)^{\frac{1}{2}\max\{ 3, (1-\suprho)^{-1}\}+\varepsilon}$ then, as $x \to \infty$,
	\begin{equation}\label{eq:before32}
	\Psi(x,y)\sim x\rho(u)K(\secondsad-1) G(\secondsad,y) \sim x\rho(u)K(\alpha-1)G(\alpha,y).
\end{equation}
\end{cor}
If $\suprho=3/4$ then Corollary \ref{cor:strip} implies that, as $x \to \infty$, \eqref{eq:before32} holds for $y \ge (\log x)^{2+\varepsilon}$. Under RH, \eqref{eq:before32} holds 
for $x \ge y \ge (\log x)^{3/2+\varepsilon}$. A different behavior emerges once $y\asymp (\log x)^{3/2} (\log \log x)^{-1/2}$:
\begin{cor}\label{cor:phase}
Assume $RH$. Suppose $x \ge C_{\varepsilon}$. If $\varepsilon\log x\le y \le (\log x)^{2-\varepsilon}$ then
\[  \Psi(x,y) =x\rho(u) K(\alpha-1)G(\alpha,y) \exp\big(\Theta_{\varepsilon}\big( \frac{\log^3 x}{y^2 \log y}\big)\big).\]
If $(1+\varepsilon)\log x\le y \le (\log x)^{2-\varepsilon}$ then
\[\Psi(x,y) =x\rho(u) K(\secondsad-1)G(\secondsad,y)\exp\big(-\Theta_{\varepsilon}\big( \frac{\log^3 x}{y^2 \log y}\big)\big).\]
\end{cor}
We do not explore this, but the contribution of the $k=3$ term (i.e.~cubes) to $\log G_2(\alpha,y)$ and $\log G_2(\secondsad,y)$ also undergoes a phase transition and is of size $\asymp_{\varepsilon}\log^3 x/(y^2 \log y)$ once $y \le (\log x)^{3/2-\varepsilon}$.

A classical estimate of Hildebrand and Tenenbaum \cite[Thm.~2]{HildebrandTenenbaum1986} states that
\begin{equation}\label{eq:HTineq}
\log \left(\frac{\Psi(x,y)}{x}\right) =  (1 + O_{\varepsilon}(  \exp(-(\log y)^{\frac{3}{5}-\varepsilon}) ))\log \rho(u)+O_{\varepsilon}\left( \frac{\log (u+1)}{\log y}  \right)
\end{equation}
holds for $x \ge y \ge (\log x)^{1+\varepsilon}$ (cf.~\cite[Thm.~III.5.21]{Tenenbaum2015}). 
We offer an improvement in terms of range and error.
\begin{cor}\label{cor:ineq}
Suppose $x \ge C_{\varepsilon}$. For $x \ge y \ge \log x \cdot  L(\log x)^{\varepsilon}$,
\begin{equation}\label{eq:intermsofG122}
	\log \left(\frac{\Psi(x,y)}{xK(\alpha-1)}\right) =  (1+ O_{\varepsilon}( L(y)^{-c_{\varepsilon}}))\log \rho(u) +O_{\varepsilon}\left(\frac{1}{ \log x}\right),
\end{equation}
and the same holds with $K(\alpha-1)$ replaced by $K(\secondsad-1)$.
For $\log x \cdot L(\log x)^c \ge y \ge  \varepsilon\log x$,
\begin{equation}\label{eq:intermsofG12small}
	\log \left(\frac{\Psi(x,y)}{x\rho(u) }\right) =  \left(1+\Theta_{\varepsilon}\left( \frac{\log x}{y}\right)\right)\log G_2(\alpha,y).
\end{equation}
\end{cor}
We zoom in on the behavior at $y\asymp \log x$, which was considered by Erd\H{o}s \cite{Erdos1963} (who studied $y=\log x$), Erd\H{o}s and van Lint \cite{Erdos1966}, de Bruijn \cite{debruijn1966} and Granville \cite{Granville1989}. In \cite{Granville1989} it is explained why ``the real difficulty lies in this range''.
	\begin{cor}\label{cor:t}
	Suppose $x \ge C_{\varepsilon}$. For $\varepsilon \log x \le y \le \varepsilon^{-1}\log x$ we have
		\begin{equation}\label{eq:logtalpha}
			\log \left( \frac{\Psi(x,y)}{x\rho(u)}\right) =(1+O_{\varepsilon}(L(y)^{-c}))\int_{\sqrt{y}}^{y} (-\log(1-v^{-\alpha})-v^{-\alpha}) \frac{dv}{\log v} + f(\alpha)-f(\secondsad).
		\end{equation}
		Setting $t:=y/\log x$,
		\begin{align}
			&0<f(\alpha)-f(\secondsad) =u (\log (1+t^{-1}) - (1+t^{-1})^{-1}) + O_{\varepsilon}\big( \frac{y}{\log^2 y}\big),\\
			&0<\int_{\sqrt{y}}^{y} (-\log(1-v^{-\alpha})-v^{-\alpha}) \frac{dv}{\log v}= \frac{y}{\log y} (\log (1+t^{-1})-(1+t)^{-1}) + O_{\varepsilon}\big( \frac{y}{\log^2 y}\big).
		\end{align}
	\end{cor}
For example, if $y=\log x$ then both $f(\alpha)-f(\secondsad)$ and the integral in \eqref{eq:logtalpha} are $\sim u(\log (2) - \tfrac{1}{2})$.
\begin{remark}
If $y \le \log x$ then $\Psi(x,y) \le x^{1/\log y}\zeta(1/\log y, y)= e^{O(u)}$. Since de Bruijn showed $\rho(u)=e^{-u \log(u \log u)+O(u)}$ \cite{debruijn19512}, it follows that
\[ \log (\Psi(x,y)/(x\rho(u))) = u (\log (u (\log u)/	y)+ O(1))\]
for $y \le \log x$. In particular $ \log (\Psi(x,y)/(x\rho(u)) \sim u \log ((\log x)/y)$ as $y/\log x \to 0$.
\end{remark}
Let
\[\bddfunc(s):=(\log( \zeta(s)(s-1)))' \asymp 1 \]
for $s \in [0,1]$. Recall $f''(\secondsad)=(\log y)^2 I''(\xi(u))$ and $g''(\alpha)=\phi_2(\alpha,y)$ are estimated in \eqref{eq:derivxi} and \eqref{eq:phi2}, respectively, and are $\asymp_{\varepsilon} (\log x)(\log y)$ when $y \ge \varepsilon\log x$. The following technical lemma is proved in \S\ref{sec:tech} and is used in the proofs of Corollaries \ref{cor:strip}--\ref{cor:t}.
\begin{lem}\label{lem:technical}Suppose $x \ge C_{\varepsilon}$.
	\begin{enumerate}
		\item Define quantities $E$ and $E'$ via
		\begin{align}\label{eq:alphadifforder}
			\secondsad-\alpha =\frac{\frac{G'}{G}(\alpha,y)+\bddfunc(\alpha)}{f''(\secondsad)}(1+E)=\frac{\frac{G'}{G}(\secondsad,y)+\bddfunc(\secondsad)}{g''(\alpha)}(1+E').
		\end{align}
When $x \ge y \ge \varepsilon\log x$ we have $E \ll_{\varepsilon} (|\tfrac{G'}{G}(\alpha,y)|+1)/\log x$ and $1+E \asymp_{\varepsilon} 1$. When $x \ge y \ge (1+\varepsilon) \log x$ we have $E'\ll_{\varepsilon} (|\tfrac{G'}{G}(\secondsad,y)|+1)/\log x$ and $1+E' \asymp_{\varepsilon} 1$.
		\item Define quantities $E_1$ and $E_2$ via
		\begin{align}
			g(\secondsad)-g(\alpha) &=g''(\alpha)(\secondsad-\alpha)^2 (1 + E_1)/2 \text{ if }x \ge y \ge (1+\varepsilon)\log x,\\
			f(\alpha)-f(\secondsad) &=f''(\secondsad)(\secondsad-\alpha)^2(1 + E_2)/2 \text{ if }x \ge y \ge \varepsilon\log x.
		\end{align}
		Then $E_1,E_2 \ll_{\varepsilon} (|\tfrac{G'}{G}(\alpha,y)|+1)/\log x$ and $1+E_1,1+E_2\asymp_{\varepsilon} 1$.
		\item If $x \ge y \ge \varepsilon \log x$ then
		\begin{equation}\label{eq:Buncond}
			B(x,y) = 1+ O_{\varepsilon}\bigg(  \frac{|\frac{G'}{G}(\alpha,y)|+|(\log G)^{(2)}(\alpha,y)|(\log y)^{-1}+1}{\log x} \bigg).
		\end{equation}
		Let $h(t):= (1+t^{-1})^{-1/2}$. For $x \ge y \ge 2$,
		\begin{equation}\label{eq:BxyH}
			B(x,y) = h(y/\log x) ( 1 + O( (\log(1+\bar{u}))^{-1})).
		\end{equation}
	\end{enumerate}
\end{lem}
\begin{remark}\label{rem:firstpart}
Suppose $x \ge y \ge (\log x)^{1/(1-\suprho)+\varepsilon}$ and $x \ge C_{\varepsilon}$. From Lemma \ref{lem:technical}, \eqref{eq:logGpsi2} and \eqref{eq:mtg2}, the logarithms of $G(\alpha,y)$, $G(\secondsad,y)$, $B(x,y)$ and $\secondsad/\alpha$ are 
\[\ll_{\varepsilon} \frac{1}{\log x} +\frac{(\log x)(\log \log x)}{y^{1-\suprho}} .\]
Thus, from Lemma \ref{lem:first} with the error term supplied by Proposition \ref{prop:sharpen},
\begin{equation}\label{eq:impfirstpart}
\Psi(x,y)= x \rho(u)  K(s-1) \left(1+O_{\varepsilon}\left(\frac{1}{\log x} + \frac{(\log x)(\log \log x)}{y^{1-\suprho}}\right)\right)
\end{equation}
for $s \in \{\alpha,\secondsad\}$. 
This gives a stronger version of the first part of Theorem \ref{thm:hildconjstronger}.
\end{remark}
\subsection{Proof of Corollary \ref{cor:strip}}
Our starting point is Lemma \ref{lem:first} with the error term supplied by Proposition \ref{prop:sharpen}. Since $B(x,y) \sim  1$ (by \eqref{eq:Buncond}) and $\secondsad/\alpha \sim 1$ when $y\ge (\log x)^{1+\varepsilon}$, it suffices to show that $g(\alpha)-g(\secondsad)$ and $f(\alpha)-f(\secondsad)$ are $o(1)$ as $x \to \infty$ in the considered range. By Lemma \ref{lem:technical},
	\begin{equation}\label{eq:difford}
		g(\secondsad)-g(\alpha), f(\alpha)-f(\secondsad)\asymp_{\varepsilon}  (\log x)^{-1}(\log y)^{-1} (\tfrac{G'}{G}(\alpha,y)+\bddfunc(\alpha))^2.
	\end{equation}
	By \eqref{eq:logGpsi2} and \eqref{eq:mtg2} with $i=1$,
	\begin{equation}\label{eq:lemmaswith1}
		\begin{split}
		G_1'(\alpha,y)/G_1(\alpha,y) &\ll y^{\suprho-\alpha} (\log y)^2,\\
		-G'_2(\alpha,y)/G_2(\alpha,y)  &\asymp_{\varepsilon} \int_{\sqrt{y}}^{y} t^{-2\alpha}dt \asymp \frac{y^{\max\left\{ 1-2\alpha,\frac{1}{2}-\alpha\right\}}\log y}{\max\{1,|\alpha-1/2|\log y\}} \ll (y^{1-2\alpha}+1)\log y.
	\end{split}
\end{equation}
	By Lemma \ref{lem:ytoapower}, these estimates give the result. The exponents $(1-\suprho)^{-1}/2$ and $3/2$ in the corollary arise from our bounds for $G'_1/G_1$ and $G'_2/G_2$, respectively.
\subsection{Proof of Corollary \ref{cor:phase}}
Our starting point is Lemma \ref{lem:first}. In the considered ranges, $\alpha,\secondsad \le 1/2-c_{\varepsilon}$ by \eqref{eq:alphasize} and Lemma \ref{lem:sigmastuff}. According to Lemma \ref{lem:technical}, and \eqref{eq:logGpsi2} and \eqref{eq:mtg2} with $i=1,2$, RH implies in our range that
	\[ \frac{\secondsad}{\alpha}-1, B(x,y)-1 \ll_{\varepsilon}\frac{\log^2 y}{\sqrt{y}}+ \frac{\log x}{y}  = o\left( \frac{\log^3 x}{y^2 \log y}\right).\]
	By \eqref{eq:difford} and \eqref{eq:lemmaswith1}, $g(\secondsad)-g(\alpha)$ and $f(\alpha)-f(\secondsad)$ are $\asymp_{\varepsilon}\log^3 x/(y^2 \log y)$.
\subsection{Proof of Corollary \ref{cor:ineq}}
Suppose $y \ge \varepsilon\log x$. Taking logarithms in the first part of Lemma \ref{lem:first} and using the error term supplied by Proposition \ref{prop:sharpen} we see
\begin{align}
	\log \left( \frac{\Psi(x,y)}{x \rho(u)K(\alpha-1)}\right) &=  \log G(\alpha,y)+f(\alpha)-f(\secondsad)+ \log B(x,y)+ O\bigg( \frac{1}{\alpha \log x}\bigg).
\end{align}
Recall $G = G_1 G_2$. By \eqref{eq:zerofree} with $i=0$ and Lemma \ref{lem:ytoapower},
\[\log G_1(\alpha,y) \ll y^{1-\alpha} L(y)^{-c} \asymp_{\varepsilon}  u \log (u+1) L(y)^{-c}. \]
The term $\log G_2(\alpha,y)$ is studied in \eqref{eq:mtg2}. The terms $\log B(x,y)$ and $f(\alpha)-f(\secondsad)$ are estimated in Lemma \ref{lem:technical} in terms of $G$. By \eqref{eq:zerofree} and \eqref{eq:mtg2} we find
\begin{align} 
\log B(x,y) &\ll_{\varepsilon} \frac{\log G_2(\alpha,y)}{u}  + L(y)^{-c} + \frac{1}{\log x},\\
f(\alpha)-f(\secondsad) &=\Theta_{\varepsilon} \left( \frac{\log^2 G_2(\alpha,y)}{u}\right)+ O_{\varepsilon}\left(\frac{\log G_2(\alpha,y)}{\min\{\log x,L(y)^c\}}+\frac{\log x}{L(y)^c} + \frac{1}{(\log x)(\log y)}\right).
\end{align}
Using \eqref{eq:mtg2} this gives \eqref{eq:intermsofG122} and \eqref{eq:intermsofG12small}. By \eqref{eq:alphadifforder}, \eqref{eq:intermsofG122} holds with $\secondsad$ instead of $\alpha$ too. 
\subsection{Proof of Corollary \ref{cor:t}}
We take logarithms in the first part of Lemma \ref{lem:first}. We have
	\[K(\alpha-1) \asymp \alpha^{-1} \asymp_{\varepsilon} \log y\]
	by \eqref{eq:alphasize} and
	$\log B(x,y) = O_{\varepsilon}(1)$ by \eqref{eq:BxyH}. Hence
	\[ \log \left( \frac{\Psi(x,y)}{x\rho(u)}\right) =\log G_1(\alpha,y)+\log G_2(\alpha,y) +f(\alpha)-f(\secondsad) + O_{\varepsilon}(\log \log y).\]
	By \eqref{eq:zerofree} we have
	\[ \log G_1(\alpha,y) \ll y^{1-\alpha}L(y)^{-c} \asymp_{\varepsilon} y L(y)^{-c}.\]
By \eqref{eq:smalls}, the quantity $\log G_2(\alpha,y)$ is equal to the integral in \eqref{eq:logtalpha}, finishing the proof of \eqref{eq:logtalpha}. To study $f(\alpha)-f(\secondsad)$ we use \cite[Eq.~(7.6)]{HildebrandTenenbaum1986} and \cite[p.~88]{Alladi1987}, which say
\begin{equation}\label{eq:alphabetalogscale}
\alpha=\frac{\log\big( 1 + \frac{y}{\log x}\big)}{\log y}\bigg(1+O\bigg(\frac{1}{\log y}\bigg)\bigg),\quad \xi=\log(u\log u) + \frac{\log \log u}{\log u}+O\bigg(\frac{1}{\log u}\bigg), 
\end{equation}
to deduce that
	\begin{align}\label{eq:diffscalelog}
	\alpha - \secondsad = \frac{\log \left(1+\frac{\log x}{y}\right)}{\log y}+ O_{\varepsilon}\left( \frac{1}{\log^2 y}\right)>0.
	\end{align}
By the definition of $f$,
	\begin{equation}\label{eq:difff} f(\alpha)-f(\secondsad) = (\alpha-\secondsad)\log x +\int_{\xi}^{(1-\alpha)\log y} \frac{e^t-1}{t}dt.
	\end{equation}
	The term $\alpha-\secondsad$ was just estimated. We estimate the integral in \eqref{eq:difff} using \eqref{eq:alphabetalogscale}--\eqref{eq:diffscalelog}:
	\begin{align}
		\int_{\xi}^{(1-\alpha)\log y} \frac{e^t-1}{t}dt &= \frac{1}{(1-\alpha)\log y} \bigg( 1+  O_{\varepsilon}\bigg( \frac{1}{\log y}\bigg)\bigg)  \int_{\xi}^{(1-\alpha)\log y} (e^t-1) dt \\
		&=\frac{1}{\log y } \bigg( 1+  O_{\varepsilon}\bigg(\frac{1}{ \log y}\bigg)\bigg)  (e^t-t) \Big|^{t=(1-\alpha)\log y}_{t=\xi}\\
		&= \frac{y^{1-\alpha}-e^{\xi}}{\log y } \bigg( 1+  O_{\varepsilon}\bigg(\frac{1}{ \log y}\bigg)\bigg)
	\end{align}
	and
	\begin{align}
		y^{1-\alpha}-e^{\xi} &= e^{\xi} ( e^{\log y(\secondsad-\alpha)}-1) = -e^{\xi} \bigg(1+ \frac{\log x}{y} \bigg)^{-1}\bigg(1+O_{\varepsilon}\bigg( \frac{1}{\log y}\bigg)\bigg)\\
		&=-\log x\bigg(1+ \frac{\log x}{y} \bigg)^{-1}\bigg(1+O_{\varepsilon}\bigg( \frac{1}{\log y}\bigg)\bigg).
	\end{align}
	To estimate the integral in \eqref{eq:logtalpha} we integrate by parts, obtaining it equals
	\begin{align}
		\frac{y}{\log y} (-\log(1-y^{-\alpha})-y^{-\alpha}) \bigg(1+O_{\varepsilon}\bigg(\frac{1}{\log y}\bigg)\bigg).
	\end{align}
	We use \eqref{eq:alphabetalogscale} to write
	\[ y^{-\alpha} = \bigg(1+\frac{y}{\log x}\bigg)^{-1} \bigg(1+O_{\varepsilon}\bigg(\frac{1}{\log y}\bigg)\bigg)\]
	which gives the desired approximation for the integral.
\section{Pomerance's question}
\subsection{Proof of first part of Theorem \ref{thm:pom}}
Estimate \eqref{eq:granvilleextended} follows from \eqref{eq:intermsofG12small} upon simplifying $\log G_2$ using \eqref{eq:mtg2}. Estimate \eqref{eq:smallu} is in \eqref{eq:intermsofG122} if $x \ge C$ and $x^c \ge y$ since, by Taylor-approximating $K$ at $0$,
\begin{equation}\label{eq:Kbnd}
K(\secondsad-1) =1+ \Theta\bigg(\frac{\log(u+1)}{\log y}\bigg)
\end{equation}
for $x^c \ge y \ge (\log x)^2$. 
For $x^c \le y \le (1-\varepsilon)x$, we recall de Bruijn's approximation $\Lambda(x,y)$ \cite{debruijn1951}, defined as \[ \Lambda(x,y) := x\int_{\RR} \rho(u-v)d(\lfloor y^v \rfloor/y^v)\]
for $x \not\in \ZZ$. Integrating the definition by parts gives
\begin{equation}\label{eq:intbyparts}
	\Lambda(x,y)=x\rho(u)-\{ x\}+x\int_{0}^{u-1}(-\rho'(u-v))\{y^v\}y^{-v}dv.
\end{equation}
De Bruijn \cite{debruijn1951} proved $\Psi(x,y) = \Lambda(x,y)+O(x\exp(-C\log^{1/2} x))$ for $u=O(1)$.
Suppose $x^c \le y \le (1-\varepsilon)x$. Then the contribution of $0\le v \le c_{\varepsilon}/\log y$ to the integral in the right-hand side of \eqref{eq:intbyparts} is $\ge c_{\varepsilon}/\log x$, which yields \eqref{eq:smallu} in the remaining range.
\subsection{Proof of second part of Theorem \ref{thm:pom}}
We assume RH holds. Estimate \eqref{eq:granvilleextended} in $2\log x \le y \le (\log x)^{2-\varepsilon}$ follows from the third part of Theorem \ref{thm:hildconjstronger} and for $\varepsilon\log x \le y < 2 \log x$ it is in the first part of Theorem \ref{thm:pom}. In $x(1-\varepsilon) \ge y \ge x^c$, \eqref{eq:smallu} holds by the first part of Theorem \ref{thm:pom}. In $x^c\ge y \ge (\log x)^{2+\varepsilon}$, \eqref{eq:smallu} follows from \eqref{eq:Kbnd} and \eqref{eq:impfirstpart} since $K$ is strictly decreasing. 
It remains to deal with $(\log x)^{2-\varepsilon} \le y \le (\log x)^{2+\varepsilon}$.
In this range, Corollary \ref{cor:strip} tells us
\[ \Psi(x,y) = (1+o(1))x \rho(u) K(\secondsad-1) G(\secondsad,y), \qquad x \to \infty. \]
The asymptotic estimates for $\log G_1$ and $\log G_2$ given in \eqref{eq:logGpsi2} and \eqref{eq:mtg2} respectively, yield
\begin{equation}\label{eq:logGpsi}
	\log G(\secondsad,y) = \frac{1+o(1)}{2}\int_{\sqrt{y}}^{y} \frac{dt}{t^{2\secondsad}\log t}dt + \frac{\psi(y)-y}{y^{\secondsad}\log y}+ O\bigg( \frac{y^{\frac{1}{2}-\secondsad}}{\log y}  \bigg).
\end{equation}
We want
\begin{equation}\label{eq:wantpos}\log K(\secondsad-1) + \frac{y^{\frac{1}{2}-\secondsad}}{\log y} \left( \frac{\psi(y)-y}{\sqrt{y}}+ O(1)\right) +\frac{1+o(1)}{2}\int_{\sqrt{y}}^{y} \frac{dt}{t^{2\secondsad}\log t}dt+ o(1)
\end{equation}
to be non-negative. We show that if 
\begin{equation}\label{eq:suff}
	\liminf_{y \to \infty} \frac{\psi(y)-y}{\sqrt{y}\log y}  > L
\end{equation}
holds and $x \ge C$ then \eqref{eq:wantpos} is non-negative for $(\log x)^{3/2}\le y \le (\log x)^{3}$. We consider three cases. If $(2\secondsad-1)\log y \ge C$ then \eqref{eq:suff} implies that \eqref{eq:wantpos} is positive if $x \ge C$. If $(2\secondsad-1)\log y \le -C$ then \eqref{eq:wantpos} is positive by \eqref{eq:suff} and additionally invoking \eqref{eq:mtg2} to estimate the integral in \eqref{eq:wantpos}. The most delicate range is $(2\secondsad-1)\log y= O(1)$. Here $K(\secondsad-1) \sim K(-1/2)$. Set
\[ \secondsad= \frac{1}{2} + \frac{v}{\log y}\]
so that $v$ is bounded. We express $\log(\Psi(x,y)/(x\rho(u))$ as a function of $y$ and $v$:
\begin{equation}\label{eq:funcv}
\log\left( \frac{\Psi(x,y)}{x\rho(u)}\right) = \log K(-1/2) + e^{-v} \frac{\psi(y)-y}{\sqrt{y}\log y}  + \frac{1}{2} \int_{v}^{2v} \frac{e^{-r}}{r}dr+ o(1).
\end{equation}
If \eqref{eq:suff} holds, we find by the definition of $L$ that the right-hand side of \eqref{eq:funcv} is $> c$ for some $c>0$, if $y$ is sufficiently large. If instead
\begin{equation}
	\liminf_{y \to \infty} \frac{\psi(y)-y}{\sqrt{y}\log y}  < L
\end{equation}
then, by definition, we can find $v\in \RR$ such that if $\secondsad=1/2+v/\log y$ then the right-hand side of \eqref{eq:funcv} is $<-c$ for some $c>0$, if $y$ is sufficiently large. This finishes the proof.
\section{Study of \texorpdfstring{$G_1$}{G1}}\label{sec:G1}
\subsection{Proof of Lemma \ref{lem:osc}}\label{sec:osc}
For our purposes, given a function $A$ its Mellin transform is
\[ \{\mathcal{M}A\}(s) = \int_{1}^{\infty} A(x)x^{-s-1}dx.\]
\begin{proposition}\label{prop:mellinlogg1}
	Let $T(x) := \int_{x}^{\infty}dt/(t^2 \log t)$, which decays like $(x\log x)^{-1}$ as $x \to \infty$ and blows up as $x\to 1^+$. 
	Fix $s_0 >-2$ and $a>0$. The function $\log G_1(s_0,x) + T(x^a)$ is defined at $x=1^+$ and
	\begin{equation}\label{eq:Mlogg1a}
		\{\mathcal{M}(\log G_1(s_0,x) -T(x^a))\}(s) = \frac{1}{s} \log \frac{\zeta(s+s_0)(s+s_0-1)}{\zeta(s_0)(s_0-1)(1+s/a)}
	\end{equation}
	holds for $\Re s > \max\{1-s_0,0\}$. This transform has analytic continuation to $\Re s > \max\{\suprho- s_0,-a\}$.
\end{proposition}
\begin{proof}
We require the following identity \cite[Lem.~2.2]{Diamond}:
	\begin{equation}\label{eq:DP}
		\log \log x + \gamma=\int_{1}^{x}\frac{1-v^{-1}}{v \log v}dv-\int_{x}^{\infty}\frac{dv}{v^2 \log v}, \qquad x>1.
	\end{equation}
Write $\log G_1(s_0,x)$ as $A_1(x)-(A_2(x)+A_3(x)+A_4(x))$ where
\begin{align}
A_1(x) &= \sum_{n \le x} \frac{\Lambda(n)}{n^{s_0} \log n}, \qquad A_2(x) = \log \log x + \gamma, \\
A_3(x) &= I((1-s_0)\log x), \qquad A_4(x) = \log(\zeta(s_0)(s_0-1)).
\end{align}
	By \cite[Thm.~1.3]{MV}, the Mellin transform of $x\mapsto \sum_{m \le x} h(m)$ is $s^{-1}\sum_{n=1}^{\infty}h(n)n^{-s}$, so
	\begin{equation}\label{eq:A1mellin}
		\{\mathcal{M}A_{1}\}(s) = s^{-1}  \log \zeta(s+s_0)
	\end{equation}
for $\Re s > \max\{1-s_0,0\}$. For any constant $b$ we have $\{ \mathcal{M}b\}(s)=s^{-1}b$ and so
\begin{equation}\label{eq:A4mellin}
	\{\mathcal{M}A_{4}\}(s) =s^{-1}\log (\zeta(s_0)(s_0-1)).
\end{equation}
	For $A_3$ first suppose $s_0< 1$. We shall use the  identity \cite[Eq.~(2.3)]{Diamond}
	\begin{equation}\label{eq:idendiamond} \log \frac{z+1}{z} = \int_{1}^{\infty} t^{-z} \frac{1-t^{-1}}{t \log t}dt, \qquad \Re z >0.
	\end{equation}
To verify \eqref{eq:idendiamond} we check that both sides have the same derivative and tend to $0$ when $\Re z \to \infty$. Applying \eqref{eq:idendiamond} with $z=(s_0+s-1)/(1-s_0)$ and substituting $t=x^{1-s_0}$ we obtain
	\begin{equation}\label{eq:mellina3} \log \frac{s}{s+s_0-1} = \int_{1}^{\infty} x^{-s} \frac{x^{1-s_0}-1}{x \log x}dx= s \int_{1}^{\infty}x^{-s-1} \int_{1}^{x} \frac{v^{1-s_0}-1}{v \log v}dv dx
	\end{equation}
	for $\Re s > \max\{1- s_0,0\}$, where in the last equality we integrated by parts.
	Substituting $v=e^{u/(1-s_0)}$ we find
	\begin{equation}\label{eq:A2mellin}
		\{\mathcal{M}A_{3}\}(s) =s^{-1}\log \frac{s}{s+s_0-1}.
	\end{equation}
If $s_0=1$ then $A_3 \equiv 0$ and \eqref{eq:A2mellin} still holds. If $s_0>1$, the left-hand side of \eqref{eq:mellina3} still agrees with its right-hand side by the uniqueness principle and so \eqref{eq:A2mellin} persists.
For $A_2$, we apply \eqref{eq:DP} with $x^a$ in place of $x$, obtaining
\[ A_2(x) =  \int_{1}^{x^a} \frac{1-t^{-1}}{t \log t}dt - \int_{x^a}^{\infty}\frac{dt}{t^2 \log t}-\log a =: \widetilde{A_{2}}(x)-T(x^a) -\log a\]
where $T$ is as in the statement of the proposition. We have $\{\mathcal{M} \log a\}(s) =s^{-1} \log a$. 
By \eqref{eq:mellina3} with $s_0=2$, $s/a$ in place of $s$ and the substitution $x=u^a$, we find
\[ \{\mathcal{M}\widetilde{A_{2}}\}(s)= s^{-1} \log \frac{s+a}{s}.\]
We now sum the Mellin transforms of $A_1$, $-A_2-T(x^a)=\log a-\widetilde{A_2}$, $-A_3$ and $-A_4$. 
\end{proof}
To establish Lemma \ref{lem:osc} fix $\varepsilon>0$ and $s>-2$. Suppose that $\log G_1(s,x)<x^{\suprho-s-\varepsilon}$ (resp.~$\log G_1(s,x)>-x^{\suprho-s-\varepsilon}$) for $x\ge C_{\varepsilon,s}$. Reach contradiction by applying Landau's Oscillation Theorem \cite[Lem.~15.1]{MV} to $A(x)=2x^{\suprho-s-\varepsilon}-(\log G_1(s,x)-T(x^{s+2}))$ (resp.~$A(x)=2x^{\suprho-s-\varepsilon}+\log G_1(s,x)-T(x^{s+2})$), an eventually positive function.
\subsection{Explicit formulas for \texorpdfstring{$G_1$}{G1}}
Given $x>0$ and $s \in \CC$ we let $S_1(x,s) := {\sum_{n \le x}}' \Lambda(n)n^{-s}$, where the prime on the summation indicates that if $x$ is a prime power, the last term of the sum should be multiplied by $1/2$. Landau \cite{Landau} established an explicit formula for $S_1(x,s)$ when $\zeta(s) \neq 0$.
In \S\ref{sec:appendix} we establish a truncated version of it, stated in the lemma below. We denote by $x'$ the prime power closest to $x$ not equal to $x$, and set $\langle x \rangle = |x-x'|$.
\begin{lem}\label{lem:truncatedpowerssec2}
Suppose $\zeta(s) \neq 0$. For $x \ge 4$ and $T \ge 2+|\Im s|$ we define $R_1(x,T,s)$ by
	\begin{equation}\label{eq:S1s2}
		S_1(x,s) = \frac{x^{1-s}}{1-s} - \frac{\zeta'(s)}{\zeta(s)}- \sum_{ |\Im (\rho -s)| \le T} \frac{x^{\rho-s}}{\rho-s} + \sum_{k=1}^{\infty} \frac{x^{-2k-s}}{2k+s} +R_1(x,T,s),
	\end{equation}
where the first sum is over non-trivial zeros $\rho$ of $\zeta$.	Then
	\begin{equation}\label{eq:R1s2}
		R_1(x,T,s) \ll (\log x) x'^{- \Re s}\min\left\{1,\frac{x}{T\langle x \rangle}\right\}+ \frac{\log^2 (xT)}{T} \big( 2^{|\Re s|}x^{1-\Re s} + \frac{2^{-\Re s}}{\log x}\big).
	\end{equation}
	If $s=1$ then the term $x^{1-s}/(1-s) - \zeta'(s)/\zeta(s)$ should be interpreted as $\log x - \gamma$.
\end{lem}
Next we need the identity \cite[p.~228]{Abramowitz} \begin{equation}\label{eq:idenclassic}
	I(-z)+\int_{0}^{\infty}\frac{e^{-z-t}}{z+t}dt+\gamma+\log z=0, \qquad z \in \CC \setminus (-\infty,0],
\end{equation}
where $\log z$ is chosen to be real-valued for $z>0$ (it is in fact equivalent to \eqref{eq:DP}).
\begin{cor}\label{cor:S2s2}
	Let $S_0(x,s) := {\sum_{n \le x}}' \Lambda(n)n^{-s}/\log n$.	Suppose $s \in \CC$ has $\zeta(s+t)\neq 0$ for $t \ge 0$. For $x \ge 4$ and $T \ge 2+|\Im s|$ we define $R_0(x,T,s)$ by
	\begin{align}
		S_0(x,s)&= I((1-s)\log x)+\gamma+\log \log x + \log(\zeta(s)(s-1)) \\
		& \quad-\sum_{ |\Im (\rho -s)|\le T}\int_{0}^{\infty} \frac{x^{\rho-s-t}}{\rho-s-t}dt+ \sum_{k=1}^{\infty} \int_{0}^{\infty} \frac{x^{-2k-s-t}}{2k+s+t}dt +R_0(x,T,s)
	\end{align}
	where $\log (\zeta(z)(z-1))$ is real-valued for $z>1$ and defined on $\{ s+t: t \ge 0\}$. 
	Then
	\begin{equation}
		\label{eq:Rs2}
		R_0(x,T,s) \ll x'^{-\Re s}\min\left\{1,\frac{x}{T\langle x \rangle}\right\}+ \frac{\log^2 (xT)}{T\log x} (2^{|\Re s|}x^{1-\Re s} +2^{-\Re s}).
	\end{equation}
\end{cor}
\begin{proof}
We start with an integral identity (cf.~\cite[Prop.~1]{Saias1989}):
	\[ S_0(x,s) =\int_{0}^{\infty} {\sum_{n \le x}}' \frac{\Lambda(n)}{n^{s+t}}dt = \int_{0}^{\infty} S_1(x,s+t)dt.\]
	We integrate both sides of \eqref{eq:S1s2} along $\{ s+t: t \ge 0\}$. We may interchange sum and integral because the sum over $\rho$ is finite, while the integral of the $k$-sum converges absolutely. It remains to show
	\[  \lim_{A \to \infty}\int_{0}^{A}\bigg( \frac{x^{1-s-t}-1}{1-s-t} - \frac{(\zeta(s+t)(s+t-1))'}{\zeta(s+t)(s+t-1)}\bigg)dt=I((1-s)\log x)+ \gamma+\log ( \zeta(s)(s-1)\log x).\]
The substitution $(1-s-t)\log x=v$ allows us to evaluate the integral as \[ I((1-s)\log x)-I((1-s-A)\log x) - \log (\zeta(s+A)(s+A-1))+\log (\zeta(s)(s-1)).\]
	The required limit follows from \eqref{eq:idenclassic} with $z=(s+A-1)\log x$ ($A \to \infty$).
\end{proof}
From the definition of $\log G_1$, Lemma \ref{lem:truncatedpowerssec2} and Corollary \ref{cor:S2s2} we get
\begin{cor}\label{cor:logG1s2}
	Suppose $\zeta(s) \neq 0$. For $x \ge 4$ and $T \ge 2+|\Im s|$ we have, for $R_1$ estimated in \eqref{eq:R1s2},
	\begin{align}
		-(\log G_1)'(s,x) =  \frac{\mathbf{1}_{x \in \NN}\Lambda(x)}{2x^s} 	-\sum_{|\Im (\rho-s) |\le T}\frac{x^{\rho-s}}{\rho-s} + \sum_{k=1}^{\infty}  \frac{x^{-2k-s}}{2k+s}+R_1(x,T,s).
	\end{align}	
	Suppose further that $\zeta(s+t)\neq 0$ for $t \ge 0$. We have, for $R_0$ estimated in \eqref{eq:Rs2},
	\begin{align}
		\log G_1(s,x) =  \frac{\mathbf{1}_{x \in \NN}\Lambda(x)}{2x^s \log x} 	-\sum_{|\Im (\rho-s) |\le T}\int_{0}^{\infty} \frac{x^{\rho-s-t}}{\rho-s-t}dt + \sum_{k=1}^{\infty} \int_{0}^{\infty} \frac{x^{-2k-s-t}}{2k+s+t}dt +R_0(x,T,s).
	\end{align}
\end{cor}
Applying Cauchy's integral formula to the first part of Corollary \ref{cor:logG1s2} we get
\begin{cor}\label{cor:S1is2}
	Fix $i \ge 2$ and $a >0$. Let $x \ge 4$. Suppose that $\zeta(z)\neq 0$ for $|z-s|\le a/\log x$. Then for $T\ge 2+|\Im s|+a/\log x$ we have 	
	\begin{align}\label{eq:S1i}
		-(\log G_1)^{(i)}(s,x) &=  \frac{\mathbf{1}_{x \in \NN}\Lambda(x)(-\log x)^{i-1}}{2x^s}- \sum_{ |\Im (\rho-s)| \le T} \frac{\partial^{i-1}}{\partial s^{i-1}} \frac{x^{\rho-s}}{\rho-s} +\frac{\partial^{i-1}}{\partial s^{i-1}}\sum_{k=1}^{\infty} \frac{x^{-2k-s}}{2k+s} +R_{i}
	\end{align}
for $R_i=R_i(x,T,s)$ satisfying
\begin{align}
R_{i}(x,T,s) &\ll_{i,a} (\log x)^{i} x'^{-\Re s}\min\left\{1,\frac{x}{T\langle x \rangle}\right\}+ \frac{\log^2 (xT)(\log x)^{i-1}}{T} \big(2^{|\Re s|} x^{1-\Re s}+\frac{2^{-\Re s}}{\log x}\big).
	\end{align}
\end{cor}
\subsection{Proof of Lemma \ref{lem:logg1size}}\label{sec:estimates}
We explain $i=0$; general $i$ is similar. Under our assumptions,
\[-\int_{0}^{\infty}\frac{ x^{\rho-s-t}}{\rho-s-t}dt=-\frac{1}{\log x}\frac{x^{\rho-s}}{\rho-s} \left(1+O_{\varepsilon}\left(\frac{1}{|\rho|\log x}\right)\right)\]
for every zero of $\zeta$. 
For the first estimate we apply Corollary \ref{cor:logG1s2} with $T=L(x)^c$ and use the Vinogradov--Korobov zero-free region to bound the sum over the zeros. For the second estimate we apply Corollary \ref{cor:logG1s2} with $T=x$ and write the sum over zeros as 
\[- \sum_{ |\Im \rho| \le x} \int_{0}^{\infty} \frac{x^{\rho-s-t}}{\rho-s-t}dt = -\frac{x^{-s}}{\log x}\sum_{|\rho|\le x} \frac{x^{\rho}}{\rho}(1+O_{\varepsilon}(|\rho|^{-1})).\]
Since $\sum_{\rho} 1/|\rho|^2$ converges \cite[Thm.~10.13]{MV}, $\sum_{\rho|\le x} |x^{\rho}/\rho^2| \ll x^{\suprho}$. By Lemma \ref{lem:truncatedpowerssec2} with $s=0$, $-\sum_{|\rho|\le x} x^{\rho}/\rho =\psi(x)-x+ O(\log^2 x)$ and so
\[- \sum_{ |\Im \rho| \le x} \int_{0}^{\infty} \frac{x^{\rho-s-t}}{\rho-s-t}dt = x^{-s}(\psi(x)-x+O_{\varepsilon}(x^{\suprho}))\ll_{\varepsilon} x^{-s} x^{\suprho}\log^2 x\]
where the last inequality is Exercise 1 in \cite[p.~430]{MV}.
\subsection{Proof of Lemma \ref{lem:integralform}}\label{sec:intform}
By definition,
\begin{equation}\label{eq:def} \sum_{n \le x} \frac{\Lambda(n)}{n^s \log n} = \int_{2^-}^{x} \frac{d\psi(t)}{t^s \log t}.
\end{equation}
Through the change of variables $v\mapsto (1-s)\log t$ ($s \neq 0$ fixed),
\begin{equation}\label{eq:FTC}
I((1-s)\log x) = \int_{1}^{x} \frac{t^{1-s}-1}{t\log t}dt,
\end{equation}
which is also true for $s=1$. The second part follows from \eqref{eq:def}, \eqref{eq:FTC} and \eqref{eq:DP}. For \eqref{eq:firstcase} observe that both sides of are real-analytic functions for $s>\suprho$ (since $\psi(x)=x+O(x^{\suprho}\log^2 x)$ \cite[p.~430]{MV}) so by the uniqueness principle it suffices to consider $s>1$. We use \eqref{eq:def}, \eqref{eq:FTC} and 
\[ \log \zeta(s) = \sum_{n \ge 1} \frac{\Lambda(n)}{n^s \log n} =\int_{2^-}^{\infty} \frac{d\psi(t)}{t^s \log t}\]
to find that 
\begin{align} \log G_1(s,x) &= -\int_{x}^{\infty}\frac{d(\psi(t)-t)}{t^s \log y} + H(s,x),\\
H(s,x) &:= \int_{1}^{x} \frac{1-t^{1-s}}{t \log t}dt - \int_{x}^{\infty} \frac{1}{t^s \log t} - \log ( (s-1)\log x)-\gamma.
\end{align}
The derivative of $H$ with respect to $s$ is $0$ and $H(2,x) \equiv 0$ by \eqref{eq:DP}, so $H \equiv 0$.
\section{Study of \texorpdfstring{$G_2$}{G2}: proof of Proposition \ref{prop:g2size}}\label{sec:G2}
We have $\log G_2 = \log G_{2,1} + \log G_{2,2}$ for
\begin{equation}
	\log G_{2,1}(s,x) =\sum_{\sqrt{x} <p \le x} \frac{p^{-2s}}{2}, \qquad
	\log G_{2,2}(s,x) =\sum_{k \ge 3}\sum_{x^{1/k} <p \le x} \frac{p^{-ks}}{k}.
\end{equation}
The PNT with error term shows, via integration by parts, that for $s \in [0,1]$ we have
\begin{align}\label{eq:log21}
	\log G_{2,1}(s,x)&=\frac{1+ O( L(x)^{-c})}{2}\int_{\sqrt{x}}^{x} \frac{dt}{t^{2s}\log t}.
\end{align}
For $x \neq 0$ let $ \mathrm{Ei}(x)$ be the exponential integral, to be understood in principal value sense:
\begin{equation}\label{eq:Ei}
	\mathrm{Ei} (x) =-\int_{-x}^{\infty}e^{-t}t^{-1}dt= \int_{-\infty}^{x}e^t t^{-1}dt=e^x x^{-1}(1+O(x^{-1})).
\end{equation}
By performing the change of variables $v=(1-2s)\log t$  in \eqref{eq:log21} we see that
\begin{align}
	\int_{\sqrt{x}}^{x} \frac{dt}{t^{2s}\log t} &= \mathrm{Ei}((1-2s)\log x)- \mathrm{Ei}( \big(\tfrac{1}{2}-s)\log x) \\
	& \sim 
	\begin{cases} \frac{x^{\frac{1}{2}-s}}{ \left(s-\frac{1}{2}\right)\log x} &\text{if }(2s-1)\log x \to \infty \\ \frac{x^{1-2s}}{ (1-2s)\log x}&\text{if }(2s-1)\log x \to -\infty\end{cases}\asymp \frac{x^{\max\{1-2s,\frac{1}{2}-s\}}}{\max\{1,|s-1/2|\log x\}}
\end{align}	
for $s \in [0,1]\setminus\{1/2\}$. When $s=1/2$ the integral is $\log 2$ since $(\log \log t)'=1/(t \log t)$.
\begin{lem}\label{lem:g22}
	For $1 \ge s \ge \varepsilon/\log x$ we have
	\[ \log G_{2,2}(s,x) \ll_{\varepsilon} \frac{x^{\max\{1-3s, \frac{1}{3}-s\}}}{\max\{1,|s-1/3|\log x\}}.\]
\end{lem}
\begin{proof}
If $x \ll_{\varepsilon} 1$ then $\log G_{2,2}(s,x) \ll_{\varepsilon} 1$ so we may assume $x\ge C_{\varepsilon}$.
In the same way we showed \eqref{eq:log21}, we find that the contribution of $k=3$ to $\log G_{2,2}(s,x)$ is acceptable, so we omit this case.
	We consider the contribution of $k\ge \max\{2/s,\log_2 x\}$ (base-$2$ logarithm) to $\log G_{2,2}$. For such $k$,
	\[ \sum_{x^{1/k}<p\le x} p^{-ks} \le 2^{-ks} + \sum_{p \ge 3} p^{-ks} \ll 2^{-ks} + \int_{2}^{\infty} t^{-ks}dt  \ll 2^{-ks}.\]
	Hence
	\begin{equation}\label{eq:largek} \sum_{k \ge \max\{2/s,\log_2 x\}}\sum_{x^{1/k}<p\le x} p^{-ks}/k \ll \sum_{k\ge \max\{2/s,\log_2 x\}} 2^{-ks}/k \ll x^{-s},
	\end{equation}
	which is negligible.
	It remains to consider the contribution of $4 \le k < \max\{2/s,\log_2 x\}$ to $\log G_{2,2}$. We show that primes $p \in (x^{1/4},x]$ have an acceptable contribution. The assumption $s \ge \varepsilon/\log x$ implies $1/(1-t^{-s}) \ll_{\varepsilon} 1$ when $t \ge x^{1/4}/2$, and so
	\[ \sum_{\max\{2/s,\log_2 x\} > k \ge 4} \sum_{x^{1/4}<p \le x} p^{-ks}/k\ll \sum_{k \ge 4} \int_{x^{1/4}/2}^{x} \frac{dt}{t^{ks}k\log t} \ll_{\varepsilon}  \int_{x^{1/4}/2}^{x} \frac{dt}{t^{4s}\log t} \]
	which is acceptable. For the primes $p \in (x^{1/k},x^{1/4}]$, the bound $\pi(y) \ll y/\log y$ shows
	\begin{equation}\label{eq:smallp}
		\sum_{\max\{2/s,\log_2x\} > k \ge 4} \sum_{x^{1/k}<p \le x^{1/4}} p^{-ks}/k\ll \sum_{\max\{2/s,\log_2x\} > k \ge 4} x^{\frac{1}{4}-s}/\log x \ll_{\varepsilon} x^{\frac{1}{4}-s},
	\end{equation}
where in the last inequality we used $s \ge \varepsilon/\log x$. This is an acceptable bound.
\end{proof}
Estimate \eqref{eq:mtg2} when $i=0$ is a direct consequence of \eqref{eq:log21} and Lemma \ref{lem:g22}; $i>0$ is similar and so is omitted. 
We now assume $1/4 \ge s \ge \varepsilon/\log x$ and establish \eqref{eq:smalls}. The contribution of $k \ge \max\{2/s,\log_2 x\}$ to $\log G_2$ is $\ll x^{-s}$ as in \eqref{eq:largek}. We now consider $2 \le k < \max\{2/s,\log_2 x\}$. If $x^{1/k} < p \le \sqrt{x}$ we get a contribution of $\ll_{\varepsilon} x^{1/2-s}$ similarly to \eqref{eq:smallp}. We handle $2 \le k < \max\{2/s,\log_2 x\}$ and $\sqrt{x}< p \le x$ by the PNT and integration by parts, obtaining a contribution of
	\begin{equation}\label{eq:truncatedlog}
(1+O_{\varepsilon}(L(x)^{-c}))\int_{\sqrt{x}}^{x} \sum_{2 \le k < \max\{2/s,\log_2 x\}} \frac{t^{-ks}}{k} \frac{dt}{\log t} . 
\end{equation}
	Since $t^{s} -1 \gg_{\varepsilon} 1$ when $t \in [\sqrt{x},x]$ we may extend the $k$-sum within the integral to the range $k \ge 2$ at a cost of $\ll_{\varepsilon}\int_{\sqrt{x}}^{x} t^{-2}dt \ll 1$. 
\section{Proofs of Lemma \ref{lem:technical} and Proposition \ref{prop:sharpen}}\label{sec:tech}
\begin{lem}\label{lem:gfdouble}
	Fix $2\le k \le 5$. Suppose $x \ge C_{\varepsilon}$. Let $I$ be the closed interval with endpoints $\alpha$ and $\secondsad$. For $t \in I$,
	\begin{align}\label{eq:gkt}
		g^{(k)}(t) &\asymp_{\varepsilon} (-1)^{k} (\log x) (\log y)^{k-1} \text{  if  } x \ge y \ge (1+\varepsilon)\log x,\\
\label{eq:fkt}		f^{(k)}(t) &\asymp_{\varepsilon} (-1)^{k} (\log x) (\log y)^{k-1} \text{  if  }x \ge y \ge \varepsilon \log x.
	\end{align}
\end{lem}
\begin{proof}
Suppose $x \ge y \ge (1+\varepsilon)\log x$. As shown in Lemma 4 of \cite{HildebrandTenenbaum1986}, 
	\begin{equation}
		g^{(k)}(t) = (-1)^k\sum_{p \le y}(\log p)(p^{t}-1)^{-k}Q_{k-1}(p^{t}\log p)
\end{equation}
for a polynomial $Q_{k-1}$ of degree $k-1$ and non-negative coefficients, so $(-1)^k g^{(k)}(t)$ is positive and monotone for $t>0$. By the same lemma, $g^{(k)}(\alpha)\asymp (-1)^k(\log x)(\log y)^{k-1}$ for $x \ge y \ge \log x$.
	It remains to show $g^{(k)}(\secondsad)$ is also of order $(-1)^k(\log x)(\log y)^{k-1}$. Since $\secondsad \ge c_{\varepsilon}/\log y$ by Corollary \ref{cor:sigmasize}, the same lemma shows that
	\begin{equation}\label{eq:gk}
		(\log y)^{k-1} \frac{y^{1-\secondsad}-1}{1-\secondsad} \ll_{\varepsilon} (-1)^{k} g^{(k)}(\secondsad) \ll_{\varepsilon} (\log y)^{k-1}\sum_{p \le y} \frac{\log p}{p^{\secondsad}-1}. 
	\end{equation}
	By definition of $\secondsad$, the left-hand side is $(\log x)(\log y)^{k-1}$. The sum in the right-hand side is upper bounded in \cite[Eq.~(7.1)]{HildebrandTenenbaum1986} by
	\[ \sum_{p \le y} \frac{\log p}{p^{\secondsad}-1} \ll \frac{1}{1-y^{-\secondsad}} \int_{1}^{y} t^{-\secondsad}dt+O(1) \]
	which is $\ll_{\varepsilon} \log x$ by definition of $\secondsad$. This finishes the proof of \eqref{eq:gkt}. For $f^{(k)}$,
	\begin{align} f^{(k)}(t) = (-\log y)^k I^{(k)}((1-t)\log y).
	\end{align}
Observe $I^{(k)}(v) \asymp e^v/(v+1)$ uniformly for $v \ge 0$ \cite[Lem.~4.5]{Smida}\label{lem:smida} and $e^{v}/(v+1) \asymp u$ as long as $0\le v = \xi(u) + O(1)$. Hence, by monotonicity of $I^{(k)}$, it suffices to show $0 \le (1-t)\log y = \xi(u)+O_{\varepsilon}(1)$ 	holds for $t\in\{\alpha,\secondsad\}$. For $t=\secondsad$ it is trivial. For $t=\alpha$, $(1-\alpha)\log y = \xi(u)+O_{\varepsilon}(1)$ follows from Lemma \ref{lem:sigmastuff} so it is left to show $\alpha<1$. By definition $\sum_{p\le y} \log p/(p^{\alpha}-1)=\log x$, and at $\alpha=1$ the sum is $\log y-\gamma+o(1)$ \cite[p.~182]{MV} which is $< \log x$ when $x \ge C$, so $\alpha < 1$.
\end{proof}

\begin{cor}\label{cor:gfdoublebetter}
Suppose $x \ge C_{\varepsilon}$. Let $I$ be the closed interval with endpoints $\alpha$ and $\secondsad$. For $t \in I$,
	\begin{align}
		g''(t) &= g''(\alpha)( 1+ O_{\varepsilon}(|\alpha-\secondsad|\log y)) \text{ if }x \ge y \ge (1+\varepsilon)\log x,\\ 
\label{eq:fcorr}		f''(t) &= f''(\secondsad)( 1+ O_{\varepsilon}(|\alpha-\secondsad|\log y)) \text{ if }x \ge y \ge \varepsilon\log x.
	\end{align}
\end{cor}
\begin{proof}
	For any $t \in I$, $g''(t) = g''(\alpha) + (t-\alpha)g^{(3)}(t_2)$ for some $t_2\in I$. The estimates for $g''$ and $g^{(3)}$ in Lemma \ref{lem:gfdouble} imply the result for $g''$. 
	The proof of \eqref{eq:fcorr} is similar.
\end{proof}
\begin{lem}\label{lem:gdoubfdoub}
Suppose $x \ge C_{\varepsilon}$. Let $2 \le k \le 4$ and $s \in \{\alpha,\secondsad\}$. We have
	\begin{equation}\label{eq:gfdiffs}
		g^{(k)}(\alpha) -f^{(k)}(\secondsad) = A^{(k-1)}(s) + O_{\varepsilon} (|(\log G)^{(k)}(s,y)|+|\alpha-\secondsad|(\log x)(\log y)^k)
	\end{equation}
in the range $x\ge y \ge (1+\varepsilon)\log x$ if $s=\secondsad$ and in $x\ge y \ge \varepsilon \log x$ if $s=\alpha$.
\end{lem}
\begin{proof}
When $s=\alpha$ we write $g^{(k)}(\alpha)-f^{(k)}(\secondsad) = f^{(k)}(\alpha)- f^{(k)}(\secondsad) + (\log G)^{(k)}(\alpha,y) + A^{(k-1)}(\alpha)$
	and replace $f^{(k)}(\alpha)-f^{(k)}(\secondsad)$ by $(\alpha-\secondsad)f^{(k+1)}(t)$ for $t$ between $\alpha$ and $\secondsad$. Lemma \ref{lem:gfdouble} bounds $f^{(k+1)}(t)$. The case $s=\secondsad$ is similar.
\end{proof}
\subsection{Proof of Lemma \ref{lem:technical} -- first part}
The relations $g'(\alpha)=f'(\secondsad)=0$ can be written as
	\begin{equation}\label{eq:saddef} (-\zeta'/\zeta)(\alpha,y) = (-F_2'/F_2)(\secondsad,y) = \log x.
	\end{equation}
	Writing $\zeta(s,y)$ as $F_2(s,y)G(s,y)\zeta(s)(s-1)$, \eqref{eq:saddef} implies
	\begin{equation}\label{eq:logderdiff} (-F'_2/F_2)(\alpha,y) + (F'_2/F_2)(\secondsad,y) = (G'/G)(\alpha,y)  + \bddfunc(\alpha).
	\end{equation}
	By the mean value theorem, for some $t$ between $\alpha$ and $\secondsad$ we have
	\begin{equation}\label{eq:alphasigmamvt} (-F'_2/F_2)(\alpha,y) + (F'_2/F_2)(\secondsad,y)= (\secondsad-\alpha) f''(t).
	\end{equation}
We compare \eqref{eq:alphasigmamvt} with  \eqref{eq:logderdiff} to find $1+E=f''(\secondsad)/f''(t)$ where $E$ is defined in \eqref{eq:alphadifforder}. By \eqref{eq:fkt}, $1+E \asymp_{\varepsilon} 1$. To upper bound $|E|$ we use the estimate in Corollary \ref{cor:gfdoublebetter} to find $E\ll_{\varepsilon}|\alpha-\secondsad|\log y$. We simplify this using the bound for $|\alpha-\secondsad|$ we just derived: $1+E \asymp_{\varepsilon} 1$ implies $|\alpha-\secondsad| \ll_{\varepsilon}( |\tfrac{G'}{G}(\alpha,y)|+1)/((\log x)(\log y))$. To study $E'$ we argue similarly using the relation
		\begin{equation} 
			(\secondsad-\alpha)g''(t')=-(\zeta'/\zeta)(\alpha,y)+(\zeta'/\zeta)(\secondsad,y)=(G'/G)(\secondsad,y)+\bddfunc(\secondsad)
	\end{equation}
for some $t'$ between $\alpha$ and $\secondsad$. This shows $1+E'=g''(\alpha)/g''(t')$.
	\subsection{Proof of Lemma \ref{lem:technical} -- second part}
	We approximate $g$ at $\alpha$ using a quadratic Taylor polynomial:
	\[ g(\secondsad) - g(\alpha)= \frac{g''(\alpha)(\secondsad-\alpha)^2}{2}  \bigg(1+O_{\varepsilon}\bigg(\frac{|g^{(3)}(t)||\alpha-\secondsad|}{(\log x)(\log y)}\bigg)\bigg) \]
	for some $t$ between $\alpha$ and $\secondsad$. By Lemma \ref{lem:gfdouble}, $g^{(3)}(t)\ll_{\varepsilon} (\log x)(\log y)^2$ and we bound $|\alpha-\secondsad|$ using \eqref{eq:alphadifforder}. Alternatively, $g(\secondsad)-g(\alpha)=g''(t)(\secondsad-\alpha)^2/2$ for some $t$ between $\alpha$ and $\secondsad$, and we appeal to Lemma \ref{lem:gfdouble} with $k=2$. The same arguments work for $f(\alpha)-f(\secondsad)$. 
	\subsection{Proof of Lemma \ref{lem:technical} -- third part} The square of $B$ can be written as
	\begin{equation}\label{eq:varbyvar}
		B^2(x,y)=\frac{f''(\secondsad)}{g''(\alpha)}=1 +\frac{f''(\secondsad)-g''(\alpha)}{g''(\alpha)}.
	\end{equation}
To prove \eqref{eq:BxyH} we estimate the numerator and denominator using \eqref{eq:derivxi} and \eqref{eq:phi2}, respectively. This also shows $B(x,y)$ is bounded when $y \ge \varepsilon\log x$. We turn to \eqref{eq:Buncond}. The denominator is $\asymp_{\varepsilon} (\log x)(\log y)$ by Lemma \ref{lem:gfdouble} and the numerator is estimated in Lemma \ref{lem:gdoubfdoub} in terms of $\alpha-\secondsad$ and $(\log G)^{(2)}$. Estimating $\alpha-\secondsad$ using \eqref{eq:alphadifforder} gives \eqref{eq:Buncond}.
\subsection{Proof of Proposition \ref{prop:sharpen}}\label{sec:sharpen}
The range $2 \log x \ge y > 1+\log x$ is already in Lemma \ref{lem:first} because $\alpha \asymp 1/\log y$ in this range by \eqref{eq:alphasize}. 
If $\log y \le \sqrt{\log x}$ and $y \ge 2\log x$, we make use of the Main Theorem of Saha, Sankaranarayanan and Suzuki \cite{Saha}, which in the current range gives
	\[ \Psi(x,y) = \frac{x^{\alpha}\zeta(\alpha,y)}{\alpha\sqrt{2\pi \phi_2(\alpha,y)}} \left(1+ \frac{g^{(4)}(\alpha)}{8g^{(2)}(\alpha)^2}-\frac{5g^{(3)}(\alpha)^2}{24 g^{(2)}(\alpha)^3}+O\left( \frac{1}{\alpha \log x}\right)\right).\]
	We also use Smida's result \cite[Thm.~1]{Smida}
	\[ \rho(u) = \frac{e^{\gamma-u\xi+I(\xi)}}{\sqrt{2\pi I''(\xi(u))} }\left(1+ \frac{f^{(4)}(\secondsad)}{8f^{(2)}(\secondsad)^2} - \frac{5f^{(3)}(\secondsad)^2}{24f^{(2)}(\secondsad)^3} + O(u^{-2})\right).\]
	We divide these two estimates to get the formulas in Lemma \ref{lem:first}, with the term $1+O(u^{-1})$ replaced by 
	\[1+ \frac{1}{8}\left( \frac{g^{(4)}(\alpha)}{g^{(2)}(\alpha)^2} - \frac{f^{(4)}(\secondsad)}{f^{(2)}(\secondsad)^2}\right) -\frac{5}{24} \left( \frac{g^{(3)}(\alpha)^2}{ g^{(2)}(\alpha)^3} -\frac{f^{(3)}(\secondsad)^2}{ f^{(2)}(\secondsad)^3}\right)  +O\left(\frac{1}{\alpha \log x}\right).\]
	This is estimated in Lemma \ref{lem:gdoubfdoub} as $1+O(1/(\alpha \log x))$ once we invoke \eqref{eq:alphadifforder}, \eqref{eq:zerofree} and \eqref{eq:mtg2}.
	
	It remains to consider $\log y > \sqrt{\log x}$. By Lemma \ref{lem:technical}, \eqref{eq:zerofree} and \eqref{eq:mtg2}, the quantities $g(\alpha)-g(\secondsad)$, $f(\alpha)-f(\secondsad)$, $B(x,y)-1$ and $(\secondsad-\alpha)/\alpha$ are $O(1/\log x)$, so we need to show
	\begin{equation}\label{eq:psiKlog}
	\Psi(x,y) =x\rho(u) K(\secondsad-1)(1+O(1/\log x))= x\rho(u) K(\alpha-1)(1+O(1/\log x)).
\end{equation}
By \cite[Prop.~A.5]{GorodetskyDeBruijn}, \[\Psi(x,y)=x\rho(u)K(\widetilde{\secondsad}-1)\left(1+O\left(\frac{1}{(\log x)(\log y)} + \frac{y}{x\log x}\right)\right)\] for $\widetilde{\secondsad}:=1+\rho'(u)/(\rho(u)\log y)$. This implies the first equality in \eqref{eq:psiKlog} since $-\rho'(u)/\rho(u) \in [\xi(u),\xi(u+1)]=[\xi(u),\xi(u)+O(1/u)]$ \cite[Lem.~2]{Evertse}\cite[Lem.~1]{Hildebrand1984} and so $\widetilde{\secondsad} = \secondsad + O(1/\log x)$. By \eqref{eq:alphadifforder}, the first equality in \eqref{eq:psiKlog} implies the second.
\begin{acknowledgement}
We are grateful to Sacha Mangerel for asking us about the integer analogue of \cite{gorodetsky}, sparking this work, and to Andrew Granville for discussions on exposition. We thank the referees for a careful reading of the manuscript and useful  suggestions.
\end{acknowledgement}
\appendix

\section{Proof of Lemma \ref{lem:truncatedpowerssec2}}\label{sec:appendix}
We follow the proof of \cite[Thm.~12.5]{MV}.
We apply \cite[Cor.~5.3]{MV} with 
$\sigma_0 =  \max\{0,1-\Re s\} + 1/\log x$ and $a_n = \Lambda(n)n^{-s}$ to obtain 
\begin{align}\label{eq:effecperron}
	S_1(x,s) &= \frac{1}{2\pi i}\int_{\sigma_0-iT}^{\sigma_0+iT}  -\frac{\zeta'(s+w)}{\zeta(s+w)} x^w \frac{dw}{w}+E_s,\\
	E_s &\ll \sum_{\substack{x/2<n<2x\\ n \neq x}}\frac{\Lambda(n)}{n^{\Re s}}\min\left\{ 1,\frac{x}{T|x-n|}\right\} + \frac{x^{\sigma_0}}{T}\left( -\frac{\zeta'}{\zeta}\right)(\sigma_0+\Re s).
\end{align}
In the sum over $(x/2,2x)$ we consider separately $n=x'$ and $n \neq x'$. We find
\[ E_s \ll (\log x)x'^{-\Re s}\min\left\{1,\frac{x}{T\langle x \rangle}\right\}+\frac{2^{|\Re s|}x^{1-\Re s}\log^2 x}{T} +\frac{x^{\sigma_0}}{T} \left(-\frac{\zeta'}{\zeta}\right)(\sigma_0+\Re s).\]
We use $-(\zeta/\zeta)'(t) \asymp (t-1)^{-1}$ for $t \in (1,2]$ and $-(\zeta'/\zeta)(t) \asymp 2^{-t}$ for $t \ge 2$ to find that $E_s$ can be absorbed in $R_1(x,T,s)$. Recall $T \ge 2+|\Im s|$. 
By \cite[Lem.~12.2]{MV}, there are $T_1,T_2 \in [T,T+1]$ such that
\begin{equation}\label{eq:T1def}
	\frac{\zeta'}{\zeta}(\sigma+i \Im s-iT_2),\,\frac{\zeta'}{\zeta}(\sigma+i\Im s+iT_1) \ll \log^2 T
\end{equation}
uniformly for $-1 \le \sigma \le 2$. We extend the range of integration in \eqref{eq:effecperron} from $|\Im w| \le T$ to $-T_2 \le \Im w \le T_1$. The error we incur is at most
\[ \ll \frac{x^{\sigma_0}}{T} \left( - \frac{\zeta'}{\zeta}\right)(\sigma_0+\Re s)\]
which can be absorbed in our bound for $E_s$.
Let $K>-\Re s$ denote an odd positive integer which will be taken to $\infty$, and let $\mathcal{C}$ denote the contour consisting of three line segments, connecting $\sigma_0-iT_2$, $-K-\Re s -iT_2$, $-K-\Re s+iT_1$, $\sigma_0+iT_1$, in this order. Cauchy's residue theorem shows that
\begin{multline} \frac{1}{2\pi i}\int_{\sigma_0-iT_2}^{\sigma_0+iT_1} -\frac{\zeta'(s+w)}{\zeta(s+w)} x^w \frac{dw}{w} =\frac{x^{1-s}}{1-s} -\frac{\zeta'(s)}{\zeta(s)}\\-\sum_{-T_2<\Im(\rho-s)<T_1} \frac{x^{\rho-s}}{\rho-s}+ \sum_{1 \le k<K/2}\frac{x^{-2k-s}}{2k+s}
	+ \frac{1}{2\pi i}\int_{\mathcal{C}} -\frac{\zeta'(s+w)}{\zeta(s+w)} x^w \frac{dw}{w}
\end{multline}
if $s \neq 1$. If $s=1$, the integrand has a double pole at $w=0$ and $x^{1-s}/(1-s) - \zeta'(s)/\zeta(s)$ should be replaced with the residue $\log x - \gamma$. 
We shorten the sum over $-T_2<\Im(\rho-s) <T_1$ to one over $-T\le \Im(\rho-s) \le T$, and the incurred error is
\[\ll  \sum_{\Im(\rho-s) \in (T,T_1) \cup (-T_2,-T)} \frac{x^{1-\Re s}}{|\rho-s|} \ll \frac{x^{1-\Re s}\log T}{T} \]
which is acceptable. It remains to bound the integral over $\mathcal{C}$. To bound its horizontal parts, we consider separately three ranges of $\Re w \in [-K-\Re s,\sigma_0]$. The contribution of $\Re w \in [-1-\Re s,\min\{2-\Re s,\sigma_0\}]$ can be bounded using \eqref{eq:T1def}:
\[ \frac{1}{2\pi i} \int_{-1-\Re s+iT_1}^{\min\{2-\Re s,\sigma_0\}+iT_1} -\frac{\zeta'(s+w)}{\zeta(s+w)} x^w \frac{dw}{w} \ll  \frac{\log^2 T}{T}  \frac{x^{\min\{2-\Re s,\sigma_0\}}}{\log x},\]
and the same bound holds if $T_1$ is replaced with $-T_2$. This error is acceptable.

Next, the contribution of $\Re w \in (2-\Re s,\sigma_0]$ should only be considered if this is a non-empty interval, i.e.~when $\Re s > 2 - 1 /\log x$. In this case, we use $-\zeta'(t)/\zeta(t) \ll 2^{-t}$ ($t \ge 2$) to estimate the integral as
\[ \frac{1}{2\pi i} \int_{2-\Re s+iT_1}^{\sigma_0+iT_1} -\frac{\zeta'(s+w)}{\zeta(s+w)} x^w \frac{dw}{w} \ll  \frac{2^{-\Re s}}{T }\frac{(x/2)^{\sigma_0}}{\log x} \ll  \frac{2^{-\Re s}}{T}\]
which is acceptable.
The same bound holds if $T_1$ is replaced with $-T_2$. 
To bound the contribution of $\Re w \in [-K-\Re s,-1-\Re s]$  we make use of \cite[Lem.~12.4]{MV} which says that $(\zeta'/\zeta)(z)\ll \log(|z|+1)$
holds for all $z$ with $\Re z \le -1$ and $\min_{k \ge 1}|z+2k|\ge 1/4$, and so
\begin{equation} \frac{1}{2\pi i}\int_{-K-\Re s +iT_1}^{-1-\Re s +i T_1}-\frac{\zeta'(s+w)}{\zeta(s+w)} x^w \frac{dw}{w} \ll \int_{-K}^{-1} \log(T+|a|) \frac{x^{a-\Re s}}{T} da \ll \frac{\log T}{T} \frac{x^{-1-\Re s}}{\log x}
\end{equation}
which is acceptable. The same bound holds if $T_1$ is replaced with $-T_2$. 
The integral over the vertical part of $\mathcal{C}$ is bounded using $(\zeta'/\zeta)(z)\ll \log(|z|+1)$ again:
\begin{equation}
	\frac{1}{2\pi i} \int_{-K-\Re s-iT_2}^{-K-\Re s+iT_1} -\frac{\zeta'(s+w)}{\zeta(s+w)} x^w \frac{dw}{w}  \ll \frac{\log(KT)}{K+\Re s}x^{-K-\Re s} \int_{-T_2}^{T_1}dt \ll \frac{T\log(KT)x^{-K-\Re s}}{K+\Re s}.
\end{equation}
When we let $K$ tend to $\infty$, this bound tends to $0$.

\bibliographystyle{abbrv}
\bibliography{references}

\end{document}